\documentclass[a4paper,12pt, abstract=true]{scrartcl}

\usepackage{german}
\usepackage[ansinew]{inputenc}
\usepackage[german,english,french,USenglish]{babel}
\usepackage{latexsym,amssymb,amsfonts,bbm}%amsmath
\usepackage{theorem}
\usepackage{xcolor}
\usepackage{graphicx}
\usepackage{esint}
\usepackage{stmaryrd}
\newtheorem{dummytheorem}{Dummy-Theorem}[section]
\newcommand{\proofendsign}{$\Box$} % \rule{2mm}{2mm}
\newtheorem{definition}[dummytheorem]{Definition}
\newtheorem{lemma}[dummytheorem]{Lemma}
\newtheorem{theorem}[dummytheorem]{Theorem}

\newtheorem{corollary}[dummytheorem]{Corollary}

\newenvironment{proof}{{\noindent \bf Proof }}
 {{\hspace*{\fill}\proofendsign\par\bigskip}}

\theorembodyfont{\normalfont}
\newtheorem{remarknorm}[dummytheorem]{Remark}
\newtheorem{examplenorm}[dummytheorem]{Example}

\newcommand{\U}{\mathbb{U}}

\newcommand{\N}{\mathbb{N}}
\newcommand{\Z}{\mathbb{Z}}

\newcommand{\R}{\mathbb{R}}

\newcommand{\pr}{\mathbb{P}}
\newcommand{\ex}{\mathbb{E}}
\newcommand{\vari}{\mathbb{V}{\rm ar}}
\newcommand{\covi}{\mathbb{C}{\rm ov}}

\newcommand{\eins}{\mathbbm{1}}

\newcommand{\vatr}{{\rm V@R}}
\newcommand{\avatr}{{\rm AV@R}}

\theorembodyfont{\rm}

\begin{document}

%%%%%%%%%%%%%%%%%%%%%%%%%%%%%%%%%%%%%%%%%%%%%%%%%%%%%%%%%%%%%%%%
%%%%%%%%%%%%%%%%%%%%%%%%%%%%%%%%%%%%%%%%%%%%%%%%%%%%%%%%%%%%%%%%
%%%%%%%%%%%%%%%%%%%%%%%%%%%%%%%%%%%%%%%%%%%%%%%%%%%%%%%%%%%%%%%%

\title{A definition of qualitative robustness for general point estimators,\\ and examples}

\author{
Henryk Zähle\footnote{Department of Mathematics, Saarland University, Saarbrücken, Germany; {\tt zaehle@math.uni-sb.de}}}
\date{}%\date{\small \today}
\maketitle

\begin{abstract}
A definition of qualitative robustness for point estimators in general statistical models is proposed. Some criteria for robustness are established and applied to estimators in parametric, semiparametric, and nonparametric models. In specific nonparametric models, the proposed definition boils down to Hampel robustness. It is also explained how plug-in estimators in certain nonparametric models can be reasonably classified w.r.t.\ their degrees of robustness.
\end{abstract}


%%%%%%%%%%%%%%%%%%%%%%%%%%%%%%%%%%%%%%%%%%%%%%%%%%%%%%%%%%%%%%%%
%%%%%%%%%%%%%%%%%%%%%%%%%%%%%%%%%%%%%%%%%%%%%%%%%%%%%%%%%%%%%%%%
%%%%%%%%%%%%%%%%%%%%%%%%%%%%%%%%%%%%%%%%%%%%%%%%%%%%%%%%%%%%%%%%
%%%%%%%%%%%%%%%%%%%%%%%%%%%%%%%%%%%%%%%%%%%%%%%%%%%%%%%%%%%%%%%%
%%%%%%%%%%%%%%%%%%%%%%%%%%%%%%%%%%%%%%%%%%%%%%%%%%%%%%%%%%%%%%%%
%%%%%%%%%%%%%%%%%%%%%%%%%%%%%%%%%%%%%%%%%%%%%%%%%%%%%%%%%%%%%%%%

\newpage

\section{Introduction}\label{Introduction}

Let $(\Theta,d_\Theta)$ be a metric space, where $\Theta$ will be regarded as a parameter space. Let $(\Omega,{\cal F})$ be a measurable space, and $\pr^\theta$ be any probability measure on $(\Omega,{\cal F})$ for every $\theta\in\Theta$. The set $\Omega$ can be seen as the sample space, where the sample is drawn from $\pr^\theta$ with (unknown) $\theta\in\Theta$. As usual, the triplet $(\Omega,{\cal F},\{\pr^\theta:\theta\in\Theta\})$ will be referred to as statistical model. Further, let $(\Sigma,{\cal S})$ be a measurable space and for every $n\in\N$ let $T_n:\Theta\to\Sigma$ be any map, where $T_n$ and $\Sigma$ can be regarded as an aspect function and the state space of the aspect function, respectively. For every $n\in\N$, let $\widehat T_n:\Omega\to\Sigma$ be any $({\cal F},{\cal S})$-measurable map, which can be seen as an estimator for the aspect $T_n(\theta)$ of $\theta$. Often the sample space and the estimator can be written as
\begin{equation}\label{def estimator in our setting}
    (\Omega,{\cal F})=(E^\N,{\cal E}^{\otimes\N})\quad\mbox{and}\quad\widehat T_n(x)=\widehat T_n(x_1,\ldots,x_n)\mbox{ for all }x=(x_1,x_2,\dots)\in\Omega
\end{equation}
for some measurable space $(E,{\cal E})$, which is virtually the standard statistical setting, but this particular form will not be assumed here. Finally, let $\rho$ be any metric on the set ${\cal M}_1(\Sigma)$ of all probability measures on $(\Sigma,{\cal S})$.

The following definition proposes a notion of (qualitative) robustness for the sequence of estimators $(\widehat T_n)$ which is in line with Hampel's notion of (qualitative) robustness. Note that the aspect functions $T_n$, $n\in\N$, do not play any role in the definition. They will only occur again in Section \ref{Sec Hampel Huber}.

\begin{definition}\label{def quali rob}
For any subset $\Theta_0\subset\Theta$ we use the following terminology.

(i) The sequence $(\widehat T_n)$ is said to be $(d_\Theta,\rho)$-robust on $\Theta_0$ if for every $\theta_1\in\Theta_0$ and $\varepsilon>0$ there is some $\delta>0$ such that
\begin{equation}\label{def quali rob - eq}
    \theta_2\in\Theta_0,\quad d_\Theta(\theta_1,\theta_2)\le\delta\quad\Longrightarrow\quad \rho(\pr^{\theta_1}\circ \widehat T_n^{-1}\,,\,\pr^{\theta_2}\circ \widehat T_n^{-1})\le\varepsilon \quad\mbox{for all }n\in\N.
\end{equation}

(ii) The sequence $(\widehat T_n)$ is said to be asymptotically $(d_\Theta,\rho)$-robust on $\Theta_0$ if for every $\theta_1\in\Theta_0$ and $\varepsilon>0$ there are some $\delta>0$ and $n_0\in\N$ such that
\begin{equation}\label{def quali rob - asymptotical}
    \theta_2\in\Theta_0,\quad d_\Theta(\theta_1,\theta_2)\le\delta\quad\Longrightarrow\quad \rho(\pr^{\theta_1}\circ \widehat T_n^{-1}\,,\,\pr^{\theta_2}\circ \widehat T_n^{-1})\le\varepsilon \quad\mbox{for all }n\ge n_0.
\end{equation}

(iii) The sequence $(\widehat T_n)$ is said to be finite-sample $(d_\Theta,\rho)$-robust on $\Theta_0$ if for every $\theta_1\in\Theta_0$, $\varepsilon>0$, and $n_0\in\N$ there is some $\delta>0$ such that
\begin{equation}\label{def quali rob - small sample}
    \theta_2\in\Theta_0,\quad d_\Theta(\theta_1,\theta_2)\le\delta\quad\Longrightarrow\quad \rho(\pr^{\theta_1}\circ \widehat T_n^{-1}\,,\,\pr^{\theta_2}\circ \widehat T_n^{-1})\le\varepsilon \quad\mbox{for all }1\le n\le n_0.
\end{equation}
\end{definition}

On the one hand, Definition \ref{def quali rob} is close to Hampel's definition of robustness in the context of nonparametric statistical models. Indeed, letting specifically $\Theta:=\Theta_0:={\cal M}_1(E)$ be the set of all probability measures on some complete and separable metric space $E$ (equipped with any metric $d_\Theta$ generating the weak topology), $(\Omega,{\cal F})$ be as in (\ref{def estimator in our setting}), $\widehat T_n$ be as in (\ref{def estimator in our setting}) and invariant against permutations of the arguments, $\pr^\mu:=\mu^{\otimes\N}$ for all $\mu\in\Theta$, $\Sigma:=\R^d$, and $\rho$ be the Prohorov metric, then part (i) of Definition \ref{def quali rob} coincides with the definition of robustness as given in Section 4 of \cite{Hampel1971}. Cuevas \cite{Cuevas1988} put forward Hampel's nonparametric theory by replacing $\R^d$ by a general complete and separable metric space $\Sigma$. Kr\"atschmer et al.\ \cite{Kraetschmeretal2012,Kraetschmeretal2014} considered metrics that metrize finer topologies than the weak topology, and Z\"ahle \cite{Zaehle2014b} allowed for laws $\pr^\mu$ that are not necessarily infinite product measures (for a different approach for nonparametric estimators based on dependent observations, see \cite{Boenteatal1987,Bustos1981,Cox1981,PapantoniKazakosGray1979,StrohrieglHable2014}). The distinction between asymptotic and finite-sample robustness was implicitly also done in \cite{Cuevas1988,Hampel1971}. Huber \cite{Huber1981} and other authors (e.g.\ \cite{Kraetschmeretal2012,Kraetschmeretal2014,Mizera2010}) regraded robustness simply as asymptotic robustness. Examples for robust estimators in nonparametric statistical models range from sample trimmed means \cite{Hampel1971} to L-estimators \cite{Huber1981} to Z- and M-estimators \cite{Hampel1971,Huber1981} to R-estimators \cite{Huber1981} to support vector machines \cite{HableChristmann2011}.

On the other hand, Definition \ref{def quali rob} allows for more statistical models $(\Omega,{\cal F},\{\pr^\theta:\theta\in\Theta\})$ than the one just discussed. In many classical examples of the theory of point estimation the parameter space $\Theta$ is a subset of $\R^k$ (and not the measure space ${\cal M}_1(E)$). The underlying statistical model has indeed often the shape $(E^\N,{\cal E}^{\otimes\N},\{\pr^\theta:\theta\in\Theta\})$ for some subset $\Theta\subset\R^k$. If $\pr^\theta=\mu_\theta^{\otimes\N}$ for some $\mu_\theta\in{\cal M}_1(E)$, $\theta\in\Theta$, then this model corresponds to the standard situation where one can observe the realizations of i.i.d.\ $E$-valued random variables with distribution $\mu_\theta$ but the true $k$-dimensional parameter vector $\theta$ is unknown; this setting is known as infinite product model. Robustness of the distribution of a given estimator w.r.t.\ small changes of the underlying model associated with $\theta$ is an obvious quality criterion, but it is not unique what the ``right'' notion of robustness is. In the mentioned infinite product model, for instance, a ``change'' of the underlying model can be measured in at least two ways. First, one may measure a change of $\theta$ w.r.t.\ the Euclidean distance on $\Theta$. Second, one may measure a change of the probability measure $\mu_{\theta}$ w.r.t.\ any metric on $\{\mu_\theta:\theta\in\Theta\}$ which metrizes the relative weak topology. The former approach is not covered by Hampel's theory, but it is covered by Definition \ref{def quali rob} and seems to be more natural in the context of classical parametric models (as, for instance, the Gaussian model where $\mu_\theta:={\rm N}_{m,s^2}$ for $\theta=(m,s^2)\in\Theta:=\R\times(0,\infty)$). The latter approach basically leads to a version of Hampel's definition when regarding $\widetilde\Theta:=\{\mu_\theta:\theta\in\Theta\}$ as the parameter space. But strictly speaking this approach is neither covered by the existing literature due to the traditional assumption $\widetilde\Theta={\cal M}_1(E)$. Definition \ref{def quali rob}, on the other hand, is more flexible and makes the second approach possible too.

Apart from the situation where $\Theta\subset\R^k$ (``parametric model''), the parameter space $\Theta$ is often the product of a subset of $\R^k$ and a subset of ${\cal M}_1(E)$ (``semiparametric model''). This is the case, for instance, in some parametric regression models, ARMA models, and so on. Then a change of the underlying model should be measured by any metric which metrizes the product topology on $\Theta$. In this situation the classical definition of robustness does not apply again, but Definition \ref{def quali rob} does.

The preceding discussion shows that Definition \ref{def quali rob} is suitable not only for nonparametric statistical models but also for parametric and semiparametric statistical models. In this sense, this article treats a rather general setting and facilitates more examples than the existing literature on robustness in nonparametric statistical models.

The article is organized as follows. Section \ref{Sec Hampel Huber} provides some criteria for asymptotic and finite-sample robustness in the fashion of the celebrated Hampel theorem. Section \ref{sec examples} is devoted to examples, and Section \ref{Proofs of the results of Section Hampel Huber} provides the proofs of the results of Section \ref{Sec Hampel Huber}. In Section \ref{Sec statistical functionals}, we investigate plug-in estimators in {\em nonparametric} statistical models being more general compared to \cite{Cuevas1988,Hampel1971,Kraetschmeretal2012,Zaehle2014b}, and we classify plug-in estimators on Euclidean spaces w.r.t.\ their degrees of robustness. Section \ref{Sec exponential models} provides results on robustness for estimators in dominated {\em parametric} statistical models, and Section \ref{Sec autoregressive process} is devoted to robustness of a Yule--Walker-type estimator in the {\em semiparametric} statistical model of a linear process. The Introduction will be completed with some basic remarks on Definition \ref{def quali rob}.

\begin{remarknorm}\label{def quali rob - remark}
When the metric $\rho$ is fixed, then $(d_\Theta,\rho)$-robustness of $(\widehat T_n)$ is clearly equivalent to $(d_\Theta',\rho)$-robustness of $(\widehat T_n)$ for any other metric $d_\Theta'$ which is equivalent to $d_\Theta$.
{\hspace*{\fill}$\Diamond$\par\bigskip}
\end{remarknorm}

\begin{remarknorm}
Of course, the sequence $(\widehat T_n)$ is robust on $\Theta_0$ if and only if it is both asymptotically and finite-sample robust on $\Theta_0$, and finite-sample robustness already holds when in (\ref{def quali rob - small sample}) the phrase ``for all $1\le n\le n_0$'' is replaced by ``for $n=n_0$''. Moreover, $(d_\Theta,\rho)$-robustness of $(\widehat T_n)$ on $\Theta_0$ means that the set of mappings %is in line with Hampel's definition of robustness and means that the set of mappings %\cite[Section 4]{Hampel1971}
$$
    \big\{\Theta\,\longrightarrow\,{\cal M}_1(\Sigma),~\theta\,\longmapsto\,\pr^\theta\circ\widehat T_n^{-1}~:~ n\in\N\big\}
$$
is $(d_\Theta,\rho)$-equicontinuous on $\Theta_0$.
{\hspace*{\fill}$\Diamond$\par\bigskip}
\end{remarknorm}

\begin{remarknorm}
In Definition \ref{def quali rob}, robustness of the sequence $(\widehat T_n)$ is a property which holds on the whole set $\Theta_0$. One could also define robustness at a fixed $\theta_1\in\Theta_0$ by requiring that (\ref{def quali rob - eq}) holds only for the fixed $\theta_1$; and analogously for asymptotic and finite-sample robustness. However, from a statistician's point of view, it is somewhat unsatisfying only to know that the estimator behaves robustly at a fixed parameter. After all the true parameter is unknown. For this reason, each of the conditions (\ref{def quali rob - eq})--(\ref{def quali rob - small sample}) is required to hold for all $\theta_1\in\Theta_0$.
{\hspace*{\fill}$\Diamond$\par\bigskip}
\end{remarknorm}

\begin{remarknorm}
Obviously, $(d_{\Theta}^1,\rho)$-robustness is a stronger requirement than $(d_{\Theta}^2,\rho)$-ro\-bustness when $d_{\Theta}^1\le d_{\Theta}^2$, and $(d_\Theta,\rho_1)$-robustness is a stronger requirement than $(d_\Theta,\rho_2)$-robustness when $\rho_1\ge\rho_2$. The same is true for asymptotic and finite- sample robustness.
{\hspace*{\fill}$\Diamond$\par\bigskip}
\end{remarknorm}

%%%%%%%%%%%%%%%%%%%%%%%%%%%%%%%%%%%%%%%%%%%%%%%%%%%%%%%%%%%%%%%%
%%%%%%%%%%%%%%%%%%%%%%%%%%%%%%%%%%%%%%%%%%%%%%%%%%%%%%%%%%%%%%%%
%%%%%%%%%%%%%%%%%%%%%%%%%%%%%%%%%%%%%%%%%%%%%%%%%%%%%%%%%%%%%%%%
%%%%%%%%%%%%%%%%%%%%%%%%%%%%%%%%%%%%%%%%%%%%%%%%%%%%%%%%%%%%%%%%
%%%%%%%%%%%%%%%%%%%%%%%%%%%%%%%%%%%%%%%%%%%%%%%%%%%%%%%%%%%%%%%%
%%%%%%%%%%%%%%%%%%%%%%%%%%%%%%%%%%%%%%%%%%%%%%%%%%%%%%%%%%%%%%%%

\section{Some criteria for robustness}\label{Sec Hampel Huber}

In this section, we will present some abstract criteria for asymptotic and finite-sample robustness. Take the notation introduced in Section \ref{Introduction}. Assume that $\Sigma$ is equipped with a complete and separable metric $d_\Sigma$, and that ${\cal S}$ is the corresponding Borel $\sigma$-field. In Theorem \ref{hampel-huber generalized} the metric $\rho$ will be chosen as the Prohorov metric $\rho_{\mbox{\scriptsize{\rm P}}}$. The Prohorov metric $\rho_{\mbox{\scriptsize{\rm P}}}$ on ${\cal M}_1(\Sigma)$ is defined by
\begin{equation}\label{def p metric}
    \rho_{\mbox{\scriptsize{\rm P}}}(\mu_1,\mu_2)\,:=\,\inf\{\varepsilon>0\,:\,\mu_1[A]\le\mu_2[A^\varepsilon]+\varepsilon\mbox{ for all }A\in{\cal S}\},
\end{equation}
where $A^\varepsilon:=\{s\in\Sigma:\,\inf_{a\in A}d_\Sigma(s,a)\le\varepsilon\}$. According to Theorem 2.14 in \cite{Huber1981} the Prohorov metric $\rho_{\mbox{\scriptsize{\rm P}}}$ metrizes the weak topology on ${\cal M}_1(\Sigma)$. The proofs of the following results can be found in Section \ref{Proofs of the results of Section Hampel Huber}.

First of all we will give a criterion for asymptotic robustness. To this end we will assume that the aspect function $T_n$ and estimator $\widehat T_n$ have the representations in (\ref{composition of estimator}) below. Let $(\Upsilon,d_{\Upsilon})$ be another complete and separable metric space equipped with the corresponding Borel $\sigma$-field ${\cal U}$. Let $U:\Theta\to\Upsilon$ and $V_n:\Upsilon\to\Sigma$, $n\in\N$, be any maps, and assume that $T_n$ and $\widehat T_n$ can be represented as
\begin{equation}\label{composition of estimator}
    T_n=V_n\circ U\quad\mbox{ and }\quad\widehat T_n=V_n\circ\widehat U_n
\end{equation}
for some $({\cal F},{\cal U})$-measurable map $\widehat U_n:\Omega\to\Upsilon$ for every $n\in\N$. Due to (\ref{composition of estimator}), Theorem \ref{hampel-huber generalized} below is applicable in the following two different situations. 1) On the one hand, when $T_n=T$ for all $n\in\N$, condition (b) of Theorem \ref{hampel-huber generalized} can sometimes be shown directly for $\widehat T_n$, $T$, $d_\Sigma$ (in place of $\widehat U_n$, $U$, $d_\Upsilon$). Then one can choose $\Upsilon:=\Sigma$, $U:=T$, $\widehat U_n:=\widehat T_n$, $V_n(u):=u$. This situation occurs, for instance, in dominated parametric statistical models; see Theorem \ref{hampel-huber generalized - parametric model - asymptotically} and its proof. 2) On the other hand, an empirical plug-in estimator has the shape $\widehat T_n=V_n(\widehat m_n)$ for an empirical probability measure $\widehat m_n$ and a possibly sample-size dependent statistical functional $V_n:\Theta\rightarrow\Sigma$, where $\Theta$ is any (nonparametric) subset of ${\cal M}_1(E)$ for some metric space $E$. It is known that for complete and separable $E$ condition (b) of Theorem \ref{hampel-huber generalized} holds for $\widehat U_n:=\widehat m_n$, $U(\theta):=\theta$, $d_\Upsilon:=$\,Prohorov metric. Thus, in this case one can choose $\Upsilon:=\Theta$, $U(\theta):=\theta$, $\widehat U_n:=\widehat m_n$, $V_n:=T_n$. This situation will be studied in detail in Section \ref{Sec statistical functionals}. A similar application of Theorem \ref{hampel-huber generalized} will be discussed in Section \ref{Sec autoregressive process}.

We will use the following terminology. The sequence $(V_n)$ is said to be {\em asymptotically continuous at $u_1\in\Upsilon$ if for every $\varepsilon>0$ there are some $\delta>0$ and $n_0\in\N$ such that
$$
    u_2\in\Upsilon,\quad d_\Upsilon(u_1,u_2)\le\delta\quad\Longrightarrow\quad d_\Sigma(V_n(u_1),V_n(u_2)\le\varepsilon \quad\mbox{for all }n\ge n_0.
$$
The sequence $(V_n)$ is said to be asymptotically continuous if it is asymptotically continuous at every $u_1\in\Upsilon$. Of course, if $V_n=V$ for all $n\in\N$, then asymptotic continuity of $(V_n)$ boils down to continuity of $V$. Moreover, equicontinuity of the family $\{V_n:n\in\N\}$ clearly implies asymptotic continuity of the sequence $(V_n)$.}

\begin{theorem}\label{hampel-huber generalized}
Assume that $\Sigma$ is equipped with a complete and separable metric $d_\Sigma$, and that ${\cal S}$ is the corresponding Borel $\sigma$-field. Let $\Theta_0\subset\Theta$. Assume that $T_n$ and $\widehat T_n$ are as in (\ref{composition of estimator}) and that the following two conditions hold:
\begin{itemize}
    \item[(a)] $U|_{\Theta_0}$ is $(d_\Theta,d_\Upsilon)$-continuous, and $(V_n)$ is asymptotically $(d_\Upsilon,d_\Sigma)$-continuous.
    \item[(b)] For every $\theta_1\in\Theta_0$, $\varepsilon>0$, and $\eta>0$ there are some $\delta>0$ and $n_0\in\N$ such that
    $$
        \theta_2\in\Theta_0,\quad d_\Theta(\theta_1,\theta_2)\le\delta\quad\Longrightarrow\quad \pr^{\theta_2}\big[\,d_\Upsilon(\widehat U_n,U(\theta_2))\ge\eta\,\big]\le\varepsilon \quad\mbox{for all }n\ge n_0.
    $$
\end{itemize}
Then the sequence $(\widehat T_n)$ is asymptotically $(d_\Theta,\rho_{\mbox{\scriptsize{\rm P}}})$-robust on $\Theta_0$.
\end{theorem}

\begin{remarknorm}\label{hampel-huber generalized - remark on b}
In the case where $\Upsilon=\Sigma$, $U=T$, $\widehat U_n=\widehat T_n$ and $V_n(u)=u$ for all $n\in\N$ it suffices to consider $\eta=\varepsilon/2$ in condition (b) of Theorem \ref{hampel-huber generalized}.
{\hspace*{\fill}$\Diamond$\par\bigskip}
\end{remarknorm}

\begin{remarknorm}\label{hampel-huber generalized - remark on variance}
In view of Markov's inequality, condition (b) in Theorem \ref{hampel-huber generalized} is fulfilled when for every $\theta_1\in\Theta_0$ and $\varepsilon>0$ there are some $\delta>0$ and $n_0\in\N$ such that
\begin{equation}\label{hampel-huber generalized - proof - 2 - sufficient condition}
    \theta_2\in\Theta_0,\quad d_\Theta(\theta_1,\theta_2)\le\delta\quad\Longrightarrow\quad \ex^{\theta_2}[\,d_\Upsilon(\widehat U_n,U(\theta_2))\,]\le\varepsilon \quad\mbox{for all }n\ge n_0.
\end{equation}
In the case where $\Upsilon\subset\R$ and $\widehat U_n$ is unbiased for $U(\theta)$ on $\Theta_0$ (that is, $\ex^\theta[\widehat U_n]=U(\theta)$ for all $\theta\in\Theta_0$) for every $n\in\N$, condition (b) in Theorem \ref{hampel-huber generalized} is fulfilled when for every $\theta_1\in\Theta_0$ and $\varepsilon>0$ there are some $\delta>0$ and $n_0\in\N$ such that
\begin{equation}\label{hampel-huber generalized - proof - 2 - sufficient condition - prime}
    \theta_2\in\Theta_0,\quad d_\Theta(\theta_1,\theta_2)\le\delta\quad\Longrightarrow\quad\vari^{\theta_2}[\widehat U_n]\le\varepsilon \quad\mbox{for all }n\ge n_0.
\end{equation}
In many specific situations condition (\ref{hampel-huber generalized - proof - 2 - sufficient condition - prime}) can be easily checked.
{\hspace*{\fill}$\Diamond$\par\bigskip}
\end{remarknorm}

To some extent the following theorem provides the converse of Theorem \ref{hampel-huber generalized}.

\begin{theorem}\label{hampel-huber generalized - reversed}
Assume that $\Sigma$ is equipped with a complete and separable metric $d_\Sigma$, and that ${\cal S}$ is the corresponding Borel $\sigma$-field. Let $\rho$ be any metric that metrizes the weak topology on ${\cal M}_1(\Sigma)$. Let $\Theta_0\subset\Theta$, and assume that the following two conditions hold:
\begin{itemize}
    \item[($\alpha$)] $(\widehat T_n)$ is asymptotically $(d_\Theta,\rho)$-robust on $\Theta_0$.
    \item[($\beta$)] There exists a map $T_0:\Theta_0\rightarrow\Sigma$ such that $\lim_{n\to\infty}\,\pr^\theta[d_\Sigma(\widehat T_n,T_0(\theta))\ge\eta]=0$ for every $\theta\in\Theta_0$ and $\eta >0$.
\end{itemize}
Then $T_0$ is $(d_\Theta,d_\Sigma)$-continuous.
\end{theorem}

In the following Theorem \ref{hampel-huber generalized - finite sample} we will give a criterion for finite-sample robustness. We again assume that the estimator $\widehat T_n$ has a certain decomposition, but the decomposition is different from those in (\ref{composition of estimator}). More precisely, we will assume that for every $n\in\N$ the estimator $\widehat T_n$ can be represented as a composition
\begin{equation}\label{composition of estimator - new}
    \widehat T_n=\widehat t_n\circ\Pi_n
\end{equation}
of two measurable maps $\Pi_n:\Omega\to\Omega_n$ and $\widehat t_n:\Omega_n\to\Sigma$, where $(\Omega_n,d_n)$ is some complete and separable metric space equipped with the Borel $\sigma$-field ${\cal F}_n$. Moreover, $\rho_n$ will refer to any metric on the set ${\cal M}_1(\Omega_n)$ of all probability measures on $(\Omega_n,{\cal F}_n)$ which metrizes the weak topology. The decomposition (\ref{composition of estimator - new}) is motivated by the setting where $\Omega$ is an infinite product space, $\Pi_n$ is the projection on the first $n$ coordinates and $\widehat{t}_n$ is to some extent the estimator $\widehat T_n$ itself (provided it depends only on the first $n$ coordinates); see Example \ref{hampel-huber generalized - finite sample - example} below for more details.

\begin{theorem}\label{hampel-huber generalized - finite sample}
Assume that $\Sigma$ is equipped with a complete and separable metric $d_\Sigma$, and that ${\cal S}$ is given by the corresponding Borel $\sigma$-field. Let $\rho$ be any metric that metrizes the weak topology on ${\cal M}_1(\Sigma)$. Assume that for every $n\in\N$ the estimator $\widehat T_n$ has the decomposition (\ref{composition of estimator - new}) as described above. Let $\Theta_0\subset\Theta$, and assume that the following two conditions hold:
\begin{itemize}
    \item[(c)] $\Omega_n\ni\omega_n\mapsto\widehat t_n(\omega_n)$ is $(d_n,d_\Sigma)$-continuous for every $n\in\N$.
    \item[(d)] $\Theta_0\ni\theta\mapsto\pr^\theta\circ\Pi_n^{-1}$ is $(d_\Theta,\rho_{n})$-continuous for every $n\in\N$.
\end{itemize}
Then the sequence $(\widehat T_n)$ is finite-sample $(d_\Theta,\rho)$-robust on $\Theta_0$.
\end{theorem}

\begin{examplenorm}\label{hampel-huber generalized - finite sample - example}
To illustrate the setting of Theorem \ref{hampel-huber generalized - finite sample}, let $(E,d_E)$ be a complete and separable metric space and ${\cal E}$ be the corresponding Borel $\sigma$-field. Set $(\Omega,{\cal F}):=(E^\N,{\cal E}^{\otimes\N})$ as well as $(\Omega_n,{\cal F}_n):=(E^n,{\cal E}^{\otimes n})$ for every $n\in\N$. For every $n\in\N$, equip the $n$-fold product space $E^n:=\times_{i=1}^nE$ with the metric
\begin{equation}\label{def metric on En}
    d_{n}(x^{n},y^{n})\,:=\,d_{E^n}(x^{n},y^{n})\,:=\,\max_{1\le i\le n}\,d_E(x_i^{n},y_i^{n}),\qquad x^n,y^n\in E^n
\end{equation}
which metrizes the product topology. Since $(E,d_E)$ was assumed to be complete and separable, the same is true for $(\Omega_n,d_n)=(E^n,d_{E^n})$. Let $X_i$ be the $i$-th coordinate projection on $E^\N$ (that is, $X_i(x_1,x_2,\ldots):=x_i$), and note that $\pr^\theta\circ(X_1,\ldots,X_n)^{-1}$ is an element of ${\cal M}_1(E^n)$ for every $\theta\in\Theta$. If $\widehat T_n$ is as in (\ref{def estimator in our setting}), then we may represent $\widehat T_n$ as $\widehat t_n\circ\Pi_n$ for $\widehat t_n(x_1,\ldots,x_n):=\widehat T_n(x_1,x_2,\ldots)$ and $\Pi_n:=(X_1,\ldots,X_n)$.
{\hspace*{\fill}$\Diamond$\par\bigskip}
\end{examplenorm}

%%%%%%%%%%%%%%%%%%%%%%%%%%%%%%%%%%%%%%%%%%%%%%%%%%%%%%%%%%%%%%%%
%%%%%%%%%%%%%%%%%%%%%%%%%%%%%%%%%%%%%%%%%%%%%%%%%%%%%%%%%%%%%%%%
%%%%%%%%%%%%%%%%%%%%%%%%%%%%%%%%%%%%%%%%%%%%%%%%%%%%%%%%%%%%%%%%
%%%%%%%%%%%%%%%%%%%%%%%%%%%%%%%%%%%%%%%%%%%%%%%%%%%%%%%%%%%%%%%%
%%%%%%%%%%%%%%%%%%%%%%%%%%%%%%%%%%%%%%%%%%%%%%%%%%%%%%%%%%%%%%%%
%%%%%%%%%%%%%%%%%%%%%%%%%%%%%%%%%%%%%%%%%%%%%%%%%%%%%%%%%%%%%%%%

\section{Examples}\label{sec examples}

\subsection{Plug-in estimators in nonparametric statistical models}\label{Sec statistical functionals}

In this section we will revisit classical plug-in estimators, which are the objects in the classical literature on robustness. Theorem \ref{hampel-huber generalized - plug-in estimator} below provides a generalization of the classical Hampel theorem. For $\psi\equiv 1$ it is indeed a version of Theorem 1 in \cite{Hampel1971} and of Theorems 1--2 in \cite{Cuevas1988}. Moreover, Theorem \ref{hampel-huber generalized - plug-in estimator - strong mixing} provides a version of Hampel's theorem for strongly mixing observations. Robustness under strong mixing has been already investigated in \cite{Zaehle2014b}, but the conditions of Theorem \ref{hampel-huber generalized - plug-in estimator - strong mixing} are significantly weaker than the conditions imposed in \cite{Zaehle2014b}. Indeed, in (\ref{hampel-huber generalized - plug-in estimator - strong mixing - eq}) the right-hand side is not required to hold uniformly in $\mu_2$ on $\Theta_0$ (as in \cite{Zaehle2014b}) but only for $\mu_2$ close to $\mu_1$. Also note that Theorems \ref{hampel-huber generalized - plug-in estimator} and \ref{hampel-huber generalized - plug-in estimator - strong mixing} take into account an aspect that was not considered in \cite{Cuevas1988,Hampel1971}. The possible restriction to a subset $\Theta_0$ of the domain of the statistical functionals $T_n$, $n\in\N$, enables us to introduce a finer notion of robustness and to compare estimators w.r.t.\ their ``degrees'' of robustness; for details see Subsections \ref{sec examples - Refinement}--\ref{sec examples - Degree} below.

Let $(E,d_E)$ be a complete and separable metric space and ${\cal E}$ be the corresponding Borel $\sigma$-field. %As in Example \ref{hampel-huber generalized - finite sample - example},
Set $(\Omega,{\cal F}):=(E^\N,{\cal E}^{\otimes\N})$ and let $X_i$ be the $i$-th coordinate projection on $\Omega=E^\N$. Let $\Theta$ be any subset of ${\cal M}_1(E)$ such that $\mathfrak{E}_n\subset\Theta$ for all $n\in\N$, where $\mathfrak{E}_n:=\{\widehat m_n(x_1,\ldots,x_n):x_1,\ldots,x_n\in E\}$ is the set of all empirical probability measures $\widehat m_n(x_1,\ldots,x_n):=\frac{1}{n}\sum_{i=1}^n\delta_{x_i}$ of order $n\in\N$. For every $\mu\in\Theta$, let $\pr^\mu$ be a probability measure on $(\Omega,{\cal F})$ such that $\pr^\mu\circ X_i^{-1}=\mu$ for all $i\in\N$.  This means that the coordinate projections $X_1,X_2,\ldots$ are identically distributed random variables with distribution $\mu$ under $\pr^\mu$ for every $\mu\in\Theta$. Let $(\Sigma,d_\Sigma)$ be a complete and separable metric space and ${\cal S}$ be the corresponding Borel $\sigma$-field. Let $T_n:\Theta\to\Sigma$ be any map and assume that the map $\widehat T_n:\Omega\to\Sigma$ defined by
\begin{equation}\label{Sec statistical functionals - def}
    \widehat T_n(x)\,=\,\widehat T_n(x_1,\ldots,x_n)\,:=\,T_n(\widehat m_n(x_1,\ldots,x_n)),\quad x=(x_1,x_2,\ldots)\in\Omega
\end{equation}
is $({\cal E}^{\otimes\N},{\cal S})$-measurable for every $n\in\N$.

In \cite{Cuevas1988,Hampel1971} and many other references $\Theta$ and $d_\Theta$ were chosen to be respectively ${\cal M}_1(E)$ and any metric generating the weak topology. This implies in particular that the classical Hampel theorem yields $(d_\Theta,\rho_{\mbox{\scriptsize{\rm P}}})$-robustness of $(\widehat T_n)=(T(\widehat m_n))$ only for sequences of statistical functionals $(T_n)$ that are asymptotically continuous w.r.t.\ the weak topology (and well defined) on ${\cal M}_1(E)$. On the other hand, there are many relevant sequences of statistical functionals $(T_n)$ that are {\em not} asymptotically weakly continuous. And the distributions of the plug-in estimators of two sequences of statistical functionals that are not asymptotically weakly continuous may react quite different to changes in the underlying (marginal) distribution, just as these plug-in estimators may have quite different influence functions. For this reason the authors of \cite{Kraetschmeretal2012,Kraetschmeretal2014,Zaehle2014b} allowed for metrics $d_\Theta$ that metrizes finer topologies than the relative weak topology.

It was discussed in \cite{Kraetschmeretal2012,Kraetschmeretal2014,Zaehle2014b} that the so-called (relative) $\psi$-weak topology (cf.\ Subection \ref{sec examples - psi weak topology} below) is a suitable topology in this context. The crucial point is that many relevant sequences of statistical functionals $(T_n)$ are not asymptotically continuous w.r.t.\ the weak topology but can be shown to be asymptotically continuous w.r.t.\ the $\psi$-weak topology for some suitable $\psi$ depending on $(T_n)$. In the case where $d_\Theta^\psi$ metrizes the relative {\em $\psi$-weak} topology and $T=T_n$ is continuous w.r.t.\ the relative $\psi$-weak topology, asymptotic $(d_\Theta^\psi,\rho_{\mbox{\scriptsize{\rm P}}})$-robustness of $(\widehat T_n)$ could be proven for so-called uniformly $\psi$-integrating sets $\Theta_0$ (cf.\ Definition \ref{def of uniformly psi integrating}). Lemma \ref{weak and psi weak topology} will show that in this case one even gets (asymptotic) $(d_\Theta,\rho_{\mbox{\scriptsize{\rm P}}})$-robustness of $(\widehat T_n)$ on such $\Theta_0$ for any metric $d_\Theta$ generating the relative {\em weak} topology. Indeed, the lemma shows that the locally uniformly $\psi$-integrating sets (cf.\ Definition \ref{def of uniformly psi integrating}) are exactly those sets on which the relative $\psi$-weak topology and the relative weak topology coincide.

%%%%%%%%%%%%%%%%%%%%%%%%%%%%%%%%%%%%%%%%%%%%%%%%%%%%%%%%%%%%%%%%
%%%%%%%%%%%%%%%%%%%%%%%%%%%%%%%%%%%%%%%%%%%%%%%%%%%%%%%%%%%%%%%%
%%%%%%%%%%%%%%%%%%%%%%%%%%%%%%%%%%%%%%%%%%%%%%%%%%%%%%%%%%%%%%%%

\subsubsection{The $\psi$-weak topology and locally uniformly $\psi$-integrating sets}\label{sec examples - psi weak topology}

Let $\psi: E\to[1,\infty)$ be a continuous function. Let ${\cal M}_1^\psi(E)$ be the set of all probability measures $\mu$ on $(E,{\cal E})$ satisfying $\int\psi\,d\mu<\infty$, and $C_\psi(E)$ be the set of all continuous functions on $E$ for which $\|f/\psi\|_\infty<\infty$, where $\|\cdot\|_\infty$ is the sup-norm. The $\psi$-weak topology on ${\cal M}_1^\psi(E)$ is defined to be the coarsest topology for which all mappings $\mu\mapsto\int f\,d\mu$, $f\in C_\psi(E)$, are continuous; cf.\ Section A.6 in \cite{FoellmerSchied2011}. Clearly, the $\psi$-weak topology is finer than the weak topology, and the two topologies coincide if and only if $\psi$ is bounded. Note that $\mu_n\to\mu$ $\psi$-weakly if and only if $\int fd\mu_n\to \int fd\mu$ for all $f\in C_\psi(E)$. Moreover, the $\psi$-weak topology is metrizable by
\begin{equation}\label{Def d psi}
    d_\psi(\mu_1,\mu_2):=d_{\mbox{\scriptsize{\rm w}}}(\mu_1,\mu_2)+\Big|\int\psi\,d\mu_1-\int\psi\,d\mu_2\Big|
\end{equation}
for any metric $d_{\mbox{\scriptsize{\rm w}}}$ which metrizes the weak topology. In the following definition, $d_{\Theta_0}$ refers to any metric on $\Theta_0$ $(\subset{\cal M}_1(E)$) which metrizes the relative weak topology.

\begin{definition}\label{def of uniformly psi integrating}
A set $\Theta_0\subset{\cal M}_1(E)$ is said to be locally uniformly $\psi$-integrating if for every $\mu_1\in\Theta_0$ and $\varepsilon>0$ there exist some $\delta>0$ and $a>0$ such that
$$
    \mu_2\in\Theta_0,\quad d_{\Theta_0}(\mu_1,\mu_2)\le\delta\quad\Longrightarrow\quad \int \psi\,\eins_{\{\psi\ge a\}}\,d\mu_2\,\le\,\varepsilon.
$$
It is said to be uniformly $\psi$-integrating if for every $\varepsilon>0$ there exists some $a>0$ such that
$$
    \sup_{\mu\in\Theta_0}\int \psi\,\eins_{\{\psi\ge a\}}\,d\mu\,\le\,\varepsilon.
$$
\end{definition}

Of course, any uniformly $\psi$-integrating set $\Theta_0$ is also locally uniformly $\psi$-integrating, and any locally uniformly $\psi$-integrating set $\Theta_0$ is a subset of ${\cal M}_1^\psi(E)$. The following two lemmas characterize (locally) uniformly $\psi$-integrating sets. The equivalence of the first three conditions in Lemma \ref{characterization of relative psi weak compactness} is already known from Corollary A.47 in \cite{FoellmerSchied2011}. Recall that $E$ was assumed to be a complete and separable metric space.

\begin{lemma}\label{characterization of relative psi weak compactness}
Assume that each set $\{\psi\le n\}$, $n\in\N$, is relatively compact in $E$. Then, for any $\Theta_0\subset{\cal M}_1^\psi(E)$, the following conditions are equivalent:
\begin{itemize}
    \item[(i)] $\Theta_0$ is relatively compact in ${\cal M}_1^\psi(E)$ for the $\psi$-weak topology.
    \item[(ii)] For every $\varepsilon>0$ there is a compact subset $K\subset E$ such that $\sup_{\mu\in\Theta_0}\int\psi\eins_{K^{\sf c}}\,d\mu\le \varepsilon$.
    \item[(iii)] There is a measurable function $\phi:E\to[1,\infty]$ such that each set $\{\phi/\psi\le n\}$, $n\in\N$, is relatively compact in $E$, and such that $\sup_{\mu\in\Theta_0}\int\phi\,d\mu<\infty$.
    \item[(iv)] $\Theta_0$ is uniformly $\psi$-integrating.
\end{itemize}
\end{lemma}

\begin{remarknorm}
In Lemma \ref{characterization of relative psi weak compactness}, we still have (i)$\Leftrightarrow$(ii)$\Leftrightarrow$(iii)$\Rightarrow$(iv) when the sets $\{\psi\le n\}$, $n\in\N$, are not necessarily relatively compact in E.
{\hspace*{\fill}$\Diamond$\par\bigskip}
\end{remarknorm}

\begin{proof}{\bf (of Lemma \ref{characterization of relative psi weak compactness})}
(i)$\Leftrightarrow$(ii)$\Leftrightarrow$(iii) is already known from Corollary A.47 in \cite{FoellmerSchied2011}, where one should note that $\sup_{\mu\in\Theta_0}\int\psi\,d\mu<\infty$ is  automatically implied by (ii) because $\Theta_0$ consists of {\em probability} measures only.

(ii)$\Rightarrow$(iv): Pick $\varepsilon>0$. By assumption (ii), we can choose a compact subset $K\subset E$ such that $\sup_{\mu\in\Theta_0}\int\psi\eins_{K^{\sf c}}\,d\mu\le \varepsilon$. In particular, $\sup_{\mu\in\Theta_0}\int \psi\eins_{\{\psi\ge a\}}\,d\mu$ is bounded above by $\sup_{\mu\in\Theta_0}\int \psi\eins_{\{\psi\ge a\}}\eins_{K}\,d\mu\,+\,\varepsilon$ for every $a>0$. Since $\psi$ as a continuous function is bounded on $K$, we can choose $a>0$ such that $\eins_{\{\psi\ge a\}}\eins_{K}=0$. Thus, (iv) holds.

(iv)$\Rightarrow$(ii): Suppose that condition (ii) is violated. Then there exists some $\varepsilon>0$ such that for every compact subset $K\subset E$ we can find a probability measure $\mu_K\in\Theta_0$ such that $\int \psi\eins_{K^{\sf c}}\,d\mu_K>\varepsilon$. In particular, for every $n\in\N$ we can find a probability measure $\mu_n\in\Theta_0$ such that $\int \psi\eins_{\{\psi>n\}}\,d\mu_n>\varepsilon$ for all $n\in\N$; note that the sets $\{\psi\le n\}$, $n\in\N$, are compact in $E$ (they are relatively compact by assumption and contain all of their limit points by the continuity of $\psi$). This contradicts condition (iv).
\end{proof}

\begin{lemma}\label{weak and psi weak topology}
Let $\Theta_0\subset{\cal M}_1^\psi(E)$. Then $\Theta_0$ is locally uniformly $\psi$-integrating if and only if the relative weak topology and the relative $\psi$-weak topology on $\Theta_0$ coincide.
\end{lemma}

\begin{proof}
It is easily seen that the sets
$$
    U_\varepsilon^\psi(\mu;f_1,\ldots,f_n)\,:=\,\bigcap_{i=1}^n\Big\{\nu\in{\cal M}_1^\psi(E):\Big|\int f_i\,d\nu-\int f_i\,d\mu\Big|<\varepsilon\Big\}
$$
for $\varepsilon>0$, $n\in\N$, and $f_1,\ldots,f_n\in C_\psi(E)$ form a basis for the neighborhoods of $\mu\in{\cal M}_1^\psi(E)$ for the $\psi$-weak topology on ${\cal M}_1^\psi(E)$. It follows that the family $\U_0^\psi(\mu)$ of sets $U_{\varepsilon,0}^\psi(\mu;f_1,\ldots,f_n):=U_\varepsilon^\psi(\mu;f_1,\ldots,f_n)\cap \Theta_0$ for $\varepsilon>0$, $n\in\N$, and $f_1,\ldots,f_n\in C_\psi(E)$ provides a basis for the neighborhoods of $\mu\in\Theta_0$ for the relative $\psi$-weak topology on $\Theta_0$. In the case where $\psi\equiv 1$, we write $C_{\sf b}(E)$, $U_{\varepsilon,0}(\mu;f_1,\ldots,f_n)$, and $\U_0(\mu)$ instead of $C_\psi(E)$, $U_{\varepsilon,0}^\psi(\mu;f_1,\ldots,f_n)$, and $\U_0^\psi(\mu)$, respectively.

First, assume that $\Theta_0$ is locally uniformly $\psi$-integrating. Obviously, the relative $\psi$-weak topology is finer than the relative weak topology on $\Theta_0$. So it suffices to show that for every $\mu\in\Theta_0$ and $U^\psi\in\U_0^\psi(\mu)$ there exists some $U\in\U_0(\mu)$ such that $U\subset U^\psi$. Let $\mu\in\Theta_0$ and $U^\psi=U_{\varepsilon,0}^\psi(\mu;f_1,\ldots,f_n)\in\U_0^\psi(\mu)$. For every $i=1,\ldots,n$ there exists some constant $c_i>0$ such that $|f_i|\le c_i\,\psi$. Set $c:=\max_{1\le i\le n}c_i$. Since $\Theta_0$ is locally uniformly $\psi$-integrating, we can choose some $\delta_\varepsilon>0$ and $a_\varepsilon>0$ such that for every $\nu\in\Theta_0$ with $d_{\Theta_0}(\mu,\nu)<\delta_\varepsilon$ we have that $\int \psi\,\eins_{\{c\psi\ge a_\varepsilon\}}\,d\nu<\varepsilon/(4c)$. Set $f_{i,\varepsilon}:=f_i\eins_{\{-a_\varepsilon< f_i<a_\varepsilon\}}+a_\varepsilon\eins_{\{f_i\ge a_\varepsilon\}}-a_\varepsilon\eins_{\{f_i\le -a_\varepsilon\}}$ and note that $f_{i,\varepsilon}\in C_{\sf b}(E)$. Then, for every $\nu\in\Theta_0$ with $d_{\Theta_0}(\mu,\nu)<\delta_\varepsilon$ and $i=1,\ldots,n$,
\begin{eqnarray*}
    \Big|\int f_i\,d\nu-\int f_i\,d\mu\Big|
    & \le & \Big|\int f_{i,\varepsilon}\,d\nu-\int f_{i,\varepsilon}\,d\mu\Big|\,+\,2c\sup_{\pi\in\Theta_0:\,d_{\Theta_0}(\mu,\pi)<\delta_\varepsilon}\int \psi\eins_{\{c\psi \ge a_\varepsilon\}}\,d\pi.
\end{eqnarray*}
The latter summand is bounded above by $\varepsilon/2$ by the choice of $\delta_\varepsilon$ and $a_\varepsilon$. It follows that
$$
    U\,:=\,U_{\widetilde\varepsilon,0}(\mu;f_{1,\varepsilon},\ldots,f_{n,\varepsilon},\widetilde f_{1,\widetilde\varepsilon},\ldots,\widetilde f_{\widetilde n,\widetilde\varepsilon})\,\subset\, U_{\varepsilon,0}^\psi(\mu;f_1,\ldots,f_n)\,=\,U^\psi,
$$
where $\widetilde\varepsilon\in(0,\varepsilon/2]$, $\widetilde n\in\N$, and $\widetilde f_{1,\widetilde\varepsilon},\ldots,\widetilde f_{\widetilde n,\widetilde\varepsilon}\in C_{\sf b}(E)$ are chosen such that the neighborhood $U_{\widetilde\varepsilon,0}(\mu;\widetilde f_{1,\widetilde\varepsilon},\ldots,\widetilde f_{\widetilde n,\widetilde\varepsilon})$ of $\mu$ is contained in the open $d_{\Theta_0}$-ball around $\mu$ with radius $\delta_\varepsilon$. Since $U\in\U_0(\mu)$, we have shown that the relative weak topology and the relative $\psi$-weak topology on $\Theta_0$ coincide.

Next, assume that the relative weak topology and the relative $\psi$-weak topology on $\Theta_0$ coincide. Suppose that $\Theta_0$ is not locally uniformly $\psi$-integrating. Then we can find some $\varepsilon>0$ and $\mu,\mu_1,\mu_2,\ldots\in\Theta_0$ such that $d_{\Theta_0}(\mu_n,\mu)\to 0$ and $\int\psi\eins_{\{\psi\ge n\}}\,d\mu_n>3\varepsilon$ for all $n\in\N$. We will show that this implies that there does not exist any $\delta>0$ such that the open $d_{\Theta_0}$-ball around $\mu$ with radius $\delta>0$ is contained in $U_{\varepsilon,0}^\psi(\mu;\psi)$. This in turn implies that we cannot find any neighborhood of $\mu$ for the relative weak topology which is contained in the neighborhood $U_{\varepsilon,0}^\psi(\mu;\psi)$ of $\mu$ for the relative $\psi$-weak topology. This contradicts the assumption. Since $\int\psi\,d\mu<\infty$, we can choose $a_\varepsilon>0$ such that $\int\psi\eins_{\{\psi>a_\varepsilon\}}\,d\mu<\varepsilon$. Since $\mu\circ\psi^{-1}$ as a probability measure on the real line has at most countably many atoms, there are at most countably many different $a>0$ with $\mu[\psi=a]>0$. In particular, we may assume $\mu[\psi=a_\varepsilon]=0$. Then
\begin{eqnarray*}
    \lefteqn{\int\psi\,d\mu_n-\int\psi\,d\mu}\\
    & = & \Big(\int\psi\eins_{\{\psi> a_\varepsilon\}}\,d\mu_n-\int\psi\eins_{\{\psi> a_\varepsilon\}}\,d\mu\Big)+\Big(\int\psi\eins_{\{\psi\le a_\varepsilon\}}\,d\mu_n-\int\psi\eins_{\{\psi\le a_\varepsilon\}}\,d\mu\Big)\\
    & =: & S_1(\varepsilon,n)\,+\,S_2(\varepsilon,n).
\end{eqnarray*}
Since $d_{\Theta_0}(\mu_n,\mu)\to 0$, the Portmanteau theorem ensures that we can find some $n_\varepsilon\in\N$ such that $|S_2(\varepsilon,n)|<\varepsilon$ for all $n\ge n_\varepsilon$. By the choice of $(\mu_n)$ and $a_\varepsilon$, we can also find some $n_\varepsilon'\ge n_\varepsilon$ such that $S_1(\varepsilon,n)>2\varepsilon$ for all $n\ge n_\varepsilon'$. That is, $|\int\psi\,d\mu_n-\int\psi\,d\mu|>\varepsilon$ for all $n\ge n_\varepsilon'$. Thus, there is indeed no $\delta>0$ such that the open $d_{\Theta_0}$-ball around $\mu$ with radius $\delta>0$ is contained in $U_{\varepsilon,0}^\psi(\mu;\psi)$. This completes the proof.
\end{proof}

%%%%%%%%%%%%%%%%%%%%%%%%%%%%%%%%%%%%%%%%%%%%%%%%%%%%%%%%%%%%%%%%
%%%%%%%%%%%%%%%%%%%%%%%%%%%%%%%%%%%%%%%%%%%%%%%%%%%%%%%%%%%%%%%%
%%%%%%%%%%%%%%%%%%%%%%%%%%%%%%%%%%%%%%%%%%%%%%%%%%%%%%%%%%%%%%%%

\subsubsection{General criteria for robustness}\label{sec examples - General criteria}

Let $\psi:E\to[1,\infty)$ be a continuous function and $d_\psi$ be the metric defined in (\ref{Def d psi}). Recall that $d_\psi$ generates the $\psi$-weak topology. Let $d_\Theta$ refer to any metric which generates the relative weak topology on $\Theta$ ($\subset{\cal M}_1(E)$) and let $\widehat T_n$ be defined by (\ref{Sec statistical functionals - def}). For every $n\in\N$ let the map $\widehat t_n:E^n\to\Sigma$  be defined by
$$%\begin{equation}\label{Sec statistical functionals - n projection - def}
    \widehat t_n(x_1,\ldots,x_n)\,:=\,\widehat T_n(x_1,\ldots,x_n)\,:=\,T_n(\widehat m_n(x_1,\ldots,x_n)),\quad x=(x_1,\ldots,x_n)\in E^n
$$%\end{equation}
and let $d_n=d_{E_n}$ be the metric on $E^n$ which was defined in (\ref{def metric on En}). Moreover, let $\rho_{\mbox{\scriptsize{\rm P}}}$ be the Prohorov metric on ${\cal M}_1(\Sigma)$ as defined in (\ref{def p metric}) and $\rho$ be an arbitrary metric on ${\cal M}_1(\Sigma)$ which metrizes the weak topology.

\begin{theorem}\label{hampel-huber generalized - plug-in estimator}
Take the notation from above, and assume that $\Theta\subset{\cal M}_1^\psi(E)$ and that $\pr^\mu=\mu^{\otimes\N}$ for every $\mu\in\Theta$. Then the following three assertions hold:
\begin{itemize}
    \item[(i)] If the sequence $(T_n)$ is asymptotically $(d_\psi,d_\Sigma)$-continuous, then $(\widehat T_n)$ is asymptotically $(d_\Theta,\rho_{\mbox{\scriptsize{\rm P}}})$-robust on every locally uniformly $\psi$-integrating set $\Theta_0\subset\Theta$.
    \item[(ii)] If $\widehat t_n:E^n\to\Sigma$ is $(d_n,d_\Sigma)$-continuous for every $n\in\N$, then $(\widehat T_n)$ is finite-sample $(d_\Theta,\rho)$-robust on $\Theta$.
    \item[(iii)] Assume that $(\widehat T_n)$ is asymptotically $(d_\Theta,\rho)$-robust on some $\Theta_0\subset\Theta$ and that there is a map $T_0:\Theta_0\rightarrow\Sigma$ such that $\lim_{n\to\infty}\pr^\mu[d_\Sigma(\widehat T_n,T_0(\mu))\ge\eta]=0$ for all $\eta>0$ and $\mu\in\Theta_0$. Then $T_0$ is $(d_\Theta,d_\Sigma)$-continuous.
\end{itemize}
\end{theorem}

Note that the assumption $\pr^\mu=\mu^{\otimes\N}$ means that the observations $X_1,X_2,\ldots$ are i.i.d.\ according to $\mu$ under $\pr^\mu$, and that $d_\psi$ can be replaced by any other metric which generates the $\psi$-weak topology. Also note that for $E=\R$ and a fixed functional $T$ ($=T_n$ for all $n\in\N$), part (i) of Theorem \ref{hampel-huber generalized - plug-in estimator} with ``asymptotically $(d_\Theta,\rho_{\mbox{\scriptsize{\rm P}}})$-robust'' and ``locally uniformly'' replaced by respectively ``asymptotically $(d_\psi,\rho_{\mbox{\scriptsize{\rm P}}})$-robust'' and ``uniformly'' is basically already known from Theorem 3.2 in \cite{Kraetschmeretal2014}.

\bigskip

\begin{proof}{\bf (of Theorem \ref{hampel-huber generalized - plug-in estimator})}
(i) Let $\Theta_0$ be a  locally uniformly $\psi$-integrating set. Then, by Lemma \ref{weak and psi weak topology}, asymptotic $(d_\Theta,\rho_{\mbox{\scriptsize{\rm P}}})$-robustness on $\Theta_0$ is equivalent to asymptotic $(d_\psi,\rho_{\mbox{\scriptsize{\rm P}}})$-robustness on $\Theta_0$. To verify asymptotic $(d_\psi,\rho_{\mbox{\scriptsize{\rm P}}})$-robustness, it suffices to show that conditions (a)--(b) in Theorem \ref{hampel-huber generalized} hold for $(\Upsilon,d_\Upsilon):=(\Theta,d_\psi)$, $U(\theta):=\theta$, $V_n:\equiv T_n$, and $\widehat U_n(x_1,x_2,\ldots):=\widehat m_n(x_1,\ldots,x_n)$. Condition (a) holds by assumption and the choice of $U$. To verify condition (b) we assume without loss of generality that the metric $d_{\mbox{\scriptsize{\rm w}}}$ in (\ref{Def d psi}) is given by the Prohorov metric $d_{\mbox{\scriptsize{\rm P}}}$, i.e.\ $d_\psi(\mu_1,\mu_2)=d_{\mbox{\scriptsize{\rm P}}}(\mu_1,\mu_2)+|\int\psi\,d\mu_1-\int\psi\,d\mu_2|$. Lemma 4 in \cite{Mizera2010} says
$$%\begin{equation}\label{hampel-huber generalized - plug-in estimator - proof - 10}
    \lim_{n\to\infty}\,\sup_{\mu\in{\cal M}_1(E)}\,\pr^\mu\big[d_{\mbox{\scriptsize{\rm P}}}(\widehat m_n,\mu)\ge\eta\big]\,=\,0\quad\mbox{ for all }\eta>0.
$$%\end{equation}
So it remains to show that that for every $\mu_1\in\Theta_0$, $\varepsilon>0$, and $\eta>0$ there are some $\delta>0$ and $n_0\in\N$ such that
\begin{equation}\label{hampel-huber generalized - plug-in estimator - proof - 10}
    \mu_2\in\Theta_0,\quad d_\psi(\mu_1,\mu_2)\le\delta\quad\Longrightarrow\quad \pr^{\mu_2}\Big[\,\Big|\int\psi\,d\widehat m_n-\int\psi\,d\mu_2\Big|\ge\eta\,\Big]\le\varepsilon \quad\mbox{for all }n\ge n_0.
\end{equation}
To prove (\ref{hampel-huber generalized - plug-in estimator - proof - 10}), fix $\mu_1\in\Theta_0$, $\varepsilon>0$, and $\eta>0$. Since $\Theta_0$ is locally uniformly $\psi$-integrating, we find some $\delta>0$ and $a>0$ such that $\int\psi\eins_{\{\psi\ge a\}} d\mu_2<\min\{\eta/3;\eta\varepsilon/6\}$ for all $\mu_2\in\Theta_0$ with $d_{\mbox{\scriptsize{\rm P}}}(\mu_1,\mu_2)\le\delta$. For every $\mu_2\in\Theta_0$ with $d_{\mbox{\scriptsize{\rm P}}}(\mu_1,\mu_2)\le\delta$ we then obtain
\begin{eqnarray}
    \pr^{\mu_2}\Big[\Big|\int\psi\,d\widehat m_n-\int\psi\,d\mu_2\Big|\ge\eta\Big]
    & \le & \pr^{\mu_2}\Big[\int\psi\eins_{\{\psi\ge a\}}\,d\widehat m_n\ge\frac{\eta}{3}\Big]\nonumber\\
    & & +\,\pr^{\mu_2}\Big[\,\Big|\int\psi\eins_{\{\psi<a\}}\,d\widehat m_n-\int\psi\eins_{\{\psi<a\}}\,d\mu_2\Big|\ge\frac{\eta}{3}\,\Big]\nonumber\\
    & & +\,\pr^{\mu_2}\Big[\int\psi\eins_{\{\psi\ge a\}}\,d\mu_2\ge\frac{\eta}{3}\Big]\nonumber\\
    & =: & S_1(n,a)+S_2(n,a)+S_3(a),\nonumber
\end{eqnarray}
where $S_3(a)=0$ and $S_1(n,a)\le(3/\eta)\int \psi\eins_{\{\psi\ge a\}}d\mu_2\le\varepsilon/2$ for all $n\in\N$ (by Markov's inequality). Further, by Chebychev's inequality we can find some $n_0\in\N$ such that $S_2(n,a)\le\varepsilon/2$ for all $n\ge n_0$ (and all $\mu_2\in{\cal M}_1(E)$). This proves (\ref{hampel-huber generalized - plug-in estimator - proof - 10}) with $d_\psi$ replaced by $d_{\mbox{\scriptsize{\rm P}}}$. Since $d_{\mbox{\scriptsize{\rm P}}}\le d_\psi$, we arrive at (\ref{hampel-huber generalized - plug-in estimator - proof - 10}).

(ii) Let $(\Omega_n,{\cal F}_n):=(E^n,{\cal E}^{\otimes n})$ for every $n\in\N$. Equip the $n$-fold product space $\Omega_n:=E^n$ with the metric $d_n=d_{E^n}$ and note that the corresponding Borel $\sigma$-field coincides with ${\cal E}^{\otimes n}$. Note that $\pr^\mu\circ(X_1,\ldots,X_n)^{-1}$ is an element of ${\cal M}_1(E^n)$ for every $\mu\in\Theta$. To verify finite-sample $(d_\Theta,\rho)$-robustness, it suffices to show that conditions (c)--(d) in Theorem \ref{hampel-huber generalized - finite sample} hold for $\widehat t_n$ and $\Pi_n:=(X_1,\ldots,X_n)$. Condition (c) holds by assumption. Moreover, the mapping $\Theta\ni\mu\mapsto\pr^\mu\circ(X_1,\ldots,X_n)^{-1}=\mu^{\otimes n}$ is clearly $(d_\Theta,\rho_{n})$-continuous for every $n\in\N$. So condition (d) holds, too.

(iii) To verify that the map $T_0$ is $(d_\Theta,d_\Sigma)$-continuous, it suffices to show that conditions ($\alpha$)--($\beta$) in Theorem \ref{hampel-huber generalized - reversed} hold. But these two conditions hold by assumption. This completes the proof.
\end{proof}

It is apparent from the proof that assertion (iii) of Theorem \ref{hampel-huber generalized - plug-in estimator} still holds when not necessarily $\pr^\mu=\mu^{\otimes\N}$. To obtain some analogues of assertions (i)--(ii) for the case where not necessarily $\pr^\mu=\mu^{\otimes\N}$, we need some additional assumptions on $\pr^\mu$. Recall that the strong mixing coefficients of the sequence $(X_i)$ are defined by
$$
    \alpha_n^\mu\,:=\,\sup_{k\in\N}\,\sup_{A\in{\cal F}_{1}^k,\,B\in{\cal F}_{n+k}^{\infty}}|\pr^\mu[A\cap B]-\pr^\mu[A]\pr^\mu[B]|,\quad n\in\N,
$$
where ${\cal F}_1^k:=\sigma(X_1,\ldots,X_k)$ and ${\cal F}_m^\infty:=\sigma(X_m,X_{m+1},\ldots)$. According to Rosenblatt \cite{Rosenblatt1956}, the sequence $(X_i)$ is said to be {\em strongly mixing} (or {\em $\alpha$-mixing}) under $\pr^\mu$ if the strong mixing coefficient $\alpha_n^\mu$ converges to $0$ as $n\to\infty$. For an overview on mixing conditions, see, for instance, \cite{Bradley2005,Doukhan1994}. Examples for sets $\Theta_0$ that satisfy condition (\ref{hampel-huber generalized - plug-in estimator - strong mixing - eq}) below can be derived from Sections B and 3.1 in \cite{Zaehle2014b}. As before we assume that $\pr^\mu\circ X_i^{-1}=\mu$ for all $i\in\N$ and $\mu\in\Theta$.

\begin{theorem}\label{hampel-huber generalized - plug-in estimator - strong mixing}
Take the notation from above, and assume that $\Theta\subset{\cal M}_1^\psi(E)$. Let $\rho_{n}$ be any metric on the set ${\cal M}_1(E^n)$ of all probability measures on $(E^n,{\cal E}^{\otimes n})$ which metrizes the weak topology. Then the following three assertions hold:
\begin{itemize}
    \item[(i)] Assume that $(E,d_E)$ is in addition locally compact. Let $\Theta_0\subset\Theta$ and assume that for every $\mu_1\in\Theta_0$ and $\varepsilon>0$ there are some $\delta>0$ and $n_0\in\N$ such that
    \begin{equation}\label{hampel-huber generalized - plug-in estimator - strong mixing - eq}
        \mu_2\in\Theta_0,\quad d_\Theta(\mu_1,\mu_2)\le\delta\quad\Longrightarrow\quad \alpha_n^{\mu_2}\le\varepsilon \quad\mbox{for all }n\ge n_0.
    \end{equation}
    Then, if the sequence $(T_n)$ is asymptotically $(d_\psi,d_\Sigma)$-continuous, $(\widehat T_n)$ is asymptotically $(d_\Theta,\rho_{\mbox{\scriptsize{\rm P}}})$-robust on every locally uniformly $\psi$-integrating set $\Theta_0\subset\Theta$.
    \item[(ii)] Assume that the mapping $\mu\mapsto\pr^\mu\circ(X_1,\ldots,X_n)^{-1}$ is $(d_\Theta,\rho_{n})$-continuous on some fixed $\Theta_0\subset\Theta$ for every $n\in\N$. Then, if $\widehat t_n:E^n\to\Sigma$ is $(d_n,d_\Sigma)$-continuous for every $n\in\N$, the sequence $(\widehat T_n)$ is finite-sample $(d_\Theta,\rho)$-robust on $\Theta_0$.
    \item[(iii)] Assume that $(\widehat T_n)$ is asymptotically $(d_\Theta,\rho)$-robust on some $\Theta_0\subset\Theta$ and that there is a map $T_0:\Theta_0\rightarrow\Sigma$ such that $\lim_{n\to\infty}\pr^\mu[d_\Sigma(\widehat T_n,T_0(\mu))\ge\eta]=0$ for all $\eta>0$ and $\mu\in\Theta_0$. Then $T_0$ is $(d_\Theta,d_\Sigma)$-continuous.
\end{itemize}
\end{theorem}

\begin{proof}
We will only prove part (i); parts (ii) and (iii) can be proven exactly in the same line as parts (ii) and (iii) of Theorem \ref{hampel-huber generalized - plug-in estimator}. Since $(E,d_E)$ is a locally compact, complete and separable metric space, we can find a sequence $(f_k)$ of real-valued continuous functions on $E$ with compact support such that $d_{\mbox{\scriptsize{\rm vag}}}(\mu_1,\mu_2):=\sum_{k=1}^\infty2^{-k}(1\wedge|\int f_k\,d\mu_1-\int f_k\,d\mu_2|)$ provides a metric on $\Theta$ which metrizes the weak topology; cf.\ the proof of Theorem 31.5 in \cite{Bauer2001}. Then we can argue as in the proof of part (i) of Theorem \ref{hampel-huber generalized - plug-in estimator}, where $d_{\mbox{\scriptsize{\rm P}}}$ may be replaced by $d_{\mbox{\scriptsize{\rm vag}}}$. Condition (a) of Theorem \ref{hampel-huber generalized} holds by assumption. To verify condition (b) it suffices to show that for every $\mu_1\in\Theta_0$, $\varepsilon>0$, and $\eta>0$ there are some $\delta>0$ and $n_0\in\N$ such that
\begin{equation}\label{hampel-huber generalized - plug-in estimator - strong mixing - proof - 10}
    \mu_2\in\Theta_0,\quad d_\psi(\mu_1,\mu_2)\le\delta\quad\Longrightarrow\quad \pr^{\mu_2}\big[\,d_{\mbox{\scriptsize{\rm vag}}}(\widehat m_n,\mu_2)\ge\eta\,\big]\le\varepsilon \quad\mbox{for all }n\ge n_0
\end{equation}
and that for every $\mu_1\in\Theta_0$, $\varepsilon>0$, and $\eta>0$ there are some $\delta>0$ and $n_0\in\N$ such that
\begin{equation}\label{hampel-huber generalized - plug-in estimator - strong mixing -proof - 15}
    \mu_2\in\Theta_0,\quad d_\psi(\mu_1,\mu_2)\le\delta\quad\Longrightarrow\quad \pr^{\mu_2}\Big[\,\Big|\int\psi\,d\widehat m_n-\int\psi\,d\mu_2\Big|\ge\eta\,\Big]\le\varepsilon \quad\mbox{for all }n\ge n_0.
\end{equation}

We first verify (\ref{hampel-huber generalized - plug-in estimator - strong mixing - proof - 10}). Fix $\mu_1\in\Theta_0$, $\varepsilon>0$, and $\eta>0$. Choose $k_0=k_0(\eta)\in\N$ such that $\sum_{k=k_0}^\infty2^{-k}<\eta/2$. Then,
\begin{eqnarray}
    \pr^{\mu_2}\big[d_{\mbox{\scriptsize{\rm vag}}}(\widehat m_n,\mu_2)\ge\eta\big]
    & \le & \pr^{\mu_2}\Big[\sum_{k=1}^{k_0}\Big|\int f_k\,d\widehat m_n-\int f_k\,d\mu_2\Big|\ge\frac{\eta}{2}\Big]\nonumber\\
    & \le & \sum_{k=1}^{k_0}\pr^{\mu_2}\Big[\Big|\int f_k\,d\widehat m_n-\int f_k\,d\mu_2\Big|\ge\frac{\eta}{2k_0}\Big]\nonumber\\
    & = & \sum_{k=1}^{k_0}\pr^{\mu_2}\Big[\Big|\frac{1}{n}\sum_{i=1}^nf_k(X_i)-\ex^{\mu_2}[f_k(X_1)]\Big|\ge\frac{\eta}{2k_0}\Big]\nonumber\\
    & \le & \sum_{k=1}^{k_0}\frac{128\,k_0^2\,\|f_k\|_\infty^2}{\eta^2}\,\frac{1}{n}\sum_{i=1}^n\alpha_i^{\mu_2}\nonumber\\
    & \le & \frac{128\,k_0^3\,\max_{1\le k\le k_0}\|f_k\|_\infty^2}{\eta^2}\,\frac{1}{n}\sum_{i=1}^n\alpha_i^{\mu_2}\nonumber\\
    & =: & C(k_0,\eta)\,\frac{1}{n}\sum_{i=1}^n\alpha_i^{\mu_2},\label{hampel-huber generalized - plug-in estimator - strong mixing - proof - 20}
\end{eqnarray}
where the third from last line is a consequence of Inequality (5.1) on p.\,936 in \cite{Rio1995} (noting that the sequences $(f_k(X_i))$, $k\in\N$, have the same strong mixing coefficients as the sequence $(X_i)$). By assumption we may choose $\delta=\delta(\varepsilon)>0$ and $n_0'=n_0'(\varepsilon)\in\N$ such that for every $\mu_2\in\Theta_0$ with $d_{\mbox{\scriptsize{\rm vag}}}(\mu_1,\mu_2)\le\delta$ and $n\ge n_0'$ we have $\alpha_n^{\mu_2}\le \varepsilon/(2C(k_0,\eta))$. Since every strong mixing coefficient is bounded above by $1/4$ (cf.\ Inequality (1.9) in \cite{Bradley2005}), we can also choose some $n_0=n_0(\varepsilon)\ge n_0'$ such that for every $\mu_2\in\Theta_0$ we have that $\frac{1}{n_0}\sum_{i=1}^{n_0'}\alpha_i^{\mu_2}\le \varepsilon/(2C(k_0,\eta))$. Hence, for every $\mu_2\in\Theta_0$ with $d_{\mbox{\scriptsize{\rm vag}}}(\mu_1,\mu_2)\le\delta$ and $n\ge n_0$ we obtain
\begin{eqnarray}
    \frac{1}{n}\sum_{i=1}^n\alpha_i^{\mu_2}
    & = & \frac{1}{n}\sum_{i=1}^{n_0'}\alpha_i^{\mu_2}\,+\,\frac{1}{n}\sum_{i=n_0'+1}^n\alpha_i^{\mu_2}\nonumber\\
    & \le & \frac{1}{n_0}\sum_{i=1}^{n_0'}\alpha_i^{\mu_2}\,+\,\frac{1}{n}\sum_{i=n_0'+1}^n\frac{\varepsilon}{2C(k_0,\eta)}\nonumber\\
    & \le & \frac{\varepsilon}{2C(k_0,\eta)}\,+\,\frac{\varepsilon}{2C(k_0,\eta)}\,.\nonumber%\label{hampel-huber generalized - plug-in estimator - strong mixing - proof - 30}
\end{eqnarray}
Along with (\ref{hampel-huber generalized - plug-in estimator - strong mixing - proof - 20}) this implies (\ref{hampel-huber generalized - plug-in estimator - strong mixing - proof - 10}) with $d_\psi$ replaced by $d_{\mbox{\scriptsize{\rm vag}}}$. Since $d_{\mbox{\scriptsize{\rm vag}}}\le d_\psi$, we arrive at (\ref{hampel-huber generalized - plug-in estimator - strong mixing - proof - 10}).

Moreover, (\ref{hampel-huber generalized - plug-in estimator - strong mixing -proof - 15}) can be shown analogously to (\ref{hampel-huber generalized - plug-in estimator - proof - 10}) in the proof of Theorem \ref{hampel-huber generalized - plug-in estimator}. The only difference is in the analysis of $S_2(n,a)$. Instead of Chebychev's inequality one has to use Inequality (5.1) on p.\,936 in \cite{Rio1995}. Indeed, as in (\ref{hampel-huber generalized - plug-in estimator - strong mixing - proof - 20}) one gets
$$
    \pr^{\mu_2}\Big[\,\Big|\int\psi\eins_{\{\psi<a\}}\,d\widehat m_n-\int \psi\eins_{\{\psi<a\}}\,d\mu_2\Big|\ge\frac{\eta}{3}\Big]\,\le\,C(a,\eta)\,\frac{1}{n}\sum_{i=1}^n\alpha_i^{\mu_2},
$$
and then one can proceed as above. This completes the proof of Theorem \ref{hampel-huber generalized - plug-in estimator - strong mixing}.
\end{proof}

%%%%%%%%%%%%%%%%%%%%%%%%%%%%%%%%%%%%%%%%%%%%%%%%%%%%%%%%%%%%%%%%
%%%%%%%%%%%%%%%%%%%%%%%%%%%%%%%%%%%%%%%%%%%%%%%%%%%%%%%%%%%%%%%%
%%%%%%%%%%%%%%%%%%%%%%%%%%%%%%%%%%%%%%%%%%%%%%%%%%%%%%%%%%%%%%%%

\subsubsection{A refined notion of robustness}\label{sec examples - Refinement}

Recall from Lemma \ref{weak and psi weak topology} that the locally uniformly $\psi$-integrating sets are exactly those sets on which the relative $\psi$-weak topology and the the relative weak topology coincide. This observation motivates Definition \ref{def psi robustness} below, which proposes a refined notion of robustness. Informally, the sequence $(\widehat T_n)$ will be said to be $\psi$-robust on $\Theta$ when it is ``Hampel robust'' on every subset $\Theta_0$ of $\Theta$ on which the relative $\psi$-weak topology and the relative weak topology coincide. Since $\psi\le\widetilde\psi$ implies that every locally uniformly $\widetilde\psi$-integrating set is also locally uniformly $\psi$-integrating (in particular that $\psi$-robustness entails $\widetilde\psi$-robustness), the function $\psi$ can be seen as a sort of gauge for the robustness.

The following definition is similar to Definition 2.13 in \cite{Kraetschmeretal2014} where $E$, ``$(d_\Theta,\rho_{\mbox{\scriptsize{\rm P}}})$-robust'' and ``locally uniformly'' are replaced by respectively $\R$, ``asymptotic $(d_\psi,\rho_{\mbox{\scriptsize{\rm P}}})$-robust'' and ``uniformly''. As before, $d_\Theta$ refers to any metric on $\Theta$ which metrizes the relative weak topology and $\widehat T_n$ is defined by (\ref{Sec statistical functionals - def}).

\begin{definition}\label{def psi robustness}
Take the notation from above, let $\psi: E\to[1,\infty)$ be a continuous function, and assume $\Theta\subset{\cal M}_1^\psi(E)$. The sequence $(\widehat T_n)$ is said to be $\psi$-robust (on $\Theta$) when it is $(d_\Theta,\rho_{\mbox{\scriptsize{\rm P}}})$-robust on every locally uniformly $\psi$-integrating set $\Theta_0\subset\Theta$.
\end{definition}

As already mentioned above, $\psi$-robustness is a stronger requirement than $\widetilde\psi$-robustness when $\psi\le\widetilde\psi$. In particular, $\psi_0$-robustness is the strongest notion of robustness within the framework of Definition \ref{def psi robustness}, where $\psi_0:\equiv1$. Note that every subset of ${\cal M}_1(E)$ is uniformly $\psi_0$-integrating. In this sense, $\psi_0$-robustness can be seen as Hampel robustness \cite{Cuevas1988,Hampel1971}. The following theorem provides a simple criterion for $\psi$-robustness.

\begin{theorem}\label{hampel-huber generalized - plug-in estimator - corollary}
Take the notation from above, and let $\psi: E\to[1,\infty)$ be a continuous function. Assume that $\Theta\subset{\cal M}_1^\psi(E)$ and that $\pr^\mu=\mu^{\otimes\N}$ for every $\mu\in\Theta$. Then the following assertions hold:
\begin{itemize}
    \item[(i)] If the family $\{T_n:n\in\N\}$ is equicontinuous w.r.t.\ the $\psi$-weakly topology, then the sequence $(\widehat T_n)$ is $\psi$-robust (on $\Theta$).
    \item[(ii)] If the sequence $(\widehat T_n)$ is $\psi$-robust (on $\Theta$) and there exists a map $T:\Theta\rightarrow\Sigma$ such that $\lim_{n\to\infty}\pr^\mu[|\widehat T_n-T(\mu)|\ge\eta\big]=0$ for all $\eta>0$ and $\mu\in\Theta$, then $T$ is $\psi$-weakly continuous.
\end{itemize}
In particular, when $T_n=T$ for all $n\in\N$ and $(\widehat T_n)$ is weakly consistent for $T(\mu)$ under $\pr^\mu$ for every $\mu\in\Theta$, the sequence $(\widehat T_n)$ is $\psi$-robust (on $\Theta$) if and only if the map $T$ is $\psi$-weakly continuous.
\end{theorem}

\begin{proof}
(i) The sequence $(T_n)$ is asymptotically continuous w.r.t.\ the $\psi$-weak topology, because the family $\{T_n:n\in\N\}$ was assumed to be equicontinuous w.r.t.\ the $\psi$-weak
topology. It follows by part (i) of Theorem \ref{hampel-huber generalized - plug-in estimator} that the sequence $(\widehat T_n)$ is asymptotically $(d_\Theta,\rho_{\mbox{\scriptsize{\rm P}}})$-robust on every locally uniformly $\psi$-integrating set $\Theta_0\subset\Theta$. In view of part (ii) of Theorem \ref{hampel-huber generalized - plug-in estimator}, for finite-sample $(d_\Theta,\rho_{\mbox{\scriptsize{\rm P}}})$-robustness of $(\widehat T_n)$ (on $\Theta_0$) it suffices to show that the map $(x_1,\ldots,x_n)\mapsto \widehat t_n(x_1,\ldots,x_n):=\widehat T_n(x_1,\ldots,x_n)=T_n(\widehat m_n(x_1,\ldots,x_n))$ is $(d_n,d_\Sigma)$-continuous for every $n\in\N$, where the metric $d_n=d_{E^n}$ is as defined in (\ref{def metric on En}). The continuity of $\widehat t_n$ holds true, because $T_n$ is $(d_\psi,d_\Sigma)$-continuous by assumption and the mapping $(x_1,\ldots,x_n)\mapsto\widehat m_n(x_1,\ldots,x_n)$ is easily seen to be $(d_n,d_\psi)$-continuous, where $d_\psi$ is as in (\ref{Def d psi}). This proves part (i).

(ii) Now assume that $(\widehat T_n)$ is $\psi$-robust and that there exists a map $T:\Theta\rightarrow\Sigma$ such that $\lim_{n\to\infty}\pr^\mu[|\widehat T_n-T(\mu)|\ge\eta\big]=0$ for all $\eta>0$ and $\mu\in\Theta$. The $\psi$-robustness means that $(\widehat T_n)$ is $(d_\Theta,d_\Sigma)$-robust on every locally uniformly $\psi$-integrating set $\Theta_0\subset\Theta$. By part (iii) of Theorem \ref{hampel-huber generalized - plug-in estimator} we can conclude that $T|_{\Theta_0}$ is $(d_\Theta,d_\Sigma)$-continuous for every locally uniformly $\psi$-integrating set $\Theta_0\subset\Theta$. In the remainder we will show that this implies $\psi$-weak continuity of $T$. Let $\mu,\mu_1,\mu_2,\ldots\in\Theta$ such that $\mu_n\to\mu$ $\psi$-weakly (in particular, $\mu_n\to\mu$ weakly). We have to show that $d_\Sigma(T(\mu_n),T(\mu))\to 0$. Since $T|_{\Theta_0}$ is $(d_\Theta,d_\Sigma)$-continuous for every locally uniformly $\psi$-integrating set $\Theta_0\subset\Theta$, it suffice to show that the set $\Theta_0:=\{\mu,\mu_1,\mu_2,\ldots\}$ is (locally) uniformly $\psi$-integrating. Let $\varepsilon>0$ be fixed. Choose $a_1=a_1(\varepsilon)>0$ such that $\int\psi\eins_{\{\psi>a_1\}}\,d\mu<\varepsilon/3$. Since $\mu\circ\psi^{-1}$ as a probability measure on the real line has at most countably many atoms, there are at most countably many different $a>0$ with $\mu[\psi=a]>0$. In particular, we may assume $\mu[\psi=a_1]=0$. By the ($\psi$-) weak convergence of $\mu_n$ to $\mu$ and the Portmanteau theorem, we can find some $n_1=n_1(\varepsilon)\in\N$ such that $|\int\psi\eins_{\{\psi\le a_1\}}\,d\mu-\int\psi\eins_{\{\psi\le a_1\}}\,d\mu_n|\le\varepsilon/3$ for all $n\ge n_1$. The $\psi$-weak convergence of $\mu_n$ to $\mu$ also implies the existence of some $n_0=n_0(\varepsilon)\ge n_1$ such that $|\int\psi\,d\mu_n-\int\psi\,d\mu|\le\varepsilon/3$ for all $n\ge n_0$. Thus we have for all $n\ge n_0$,
\begin{eqnarray*}
    \lefteqn{\int\psi\eins_{\{\psi>a_1\}}\,d\mu_n}\\
    & \le & \Big|\int\psi\,d\mu_n\,-\,\int\psi\,d\mu\Big|\,+\,\int\psi\eins_{\{\psi>a_1\}}\,d\mu\,+\,\Big|\int\psi\eins_{\{\psi\le a_1\}}\,d\mu\,-\,\int\psi\eins_{\{\psi\le a_1\}}\,d\mu_n\Big|\\
    & \le & \varepsilon/3\,+\,\varepsilon/3\,+\,\varepsilon/3~=~\varepsilon.
\end{eqnarray*}
Now choose $a=a(\varepsilon,n_0)\ge a_1$ such that $\int\psi\eins_{\{\psi\ge a\}}\,d\mu_n\le\varepsilon$ for all $n=1,\ldots,n_0$. We then have $\int\psi\eins_{\{\psi\ge a\}}\,d\mu\le\varepsilon$ and $\int\psi\eins_{\{\psi\ge a\}}\,d\mu_n\le\varepsilon$ for all $n\in\N$. Hence, $\Theta_0:=\{\mu,\mu_1,\mu_2,\ldots\}$ is indeed (locally) uniformly $\psi$-integrating.
\end{proof}

When $E=\R$, the function $\psi_p:\R\rightarrow[1,\infty)$ defined by
\begin{equation}\label{def of psi p}
    \psi_p(x):=(1+|x|)^p,\quad x\in\R
\end{equation}
provides a reasonable gauge function for every $p\in[0,\infty)$. Note that ${\cal M}_1^{\psi_p}(\R)$ coincides with the set of the distributions of all random variables in the usual $L^p$-space (associated with any atomless probability space). For $0\le q<p$, we have $\psi_q<\psi_p$ (except at the origin) and therefore $\psi_q$-robustness is a stronger condition than $\psi_p$-robustness. In the following example we will see that the Average Value at Risk functional divided by $n$ and evaluated at the $n$-fold convolution of an empirical probability measure $\widehat m_n$ (based on i.i.d.\ observations) is $\psi_1$-robust but not $\psi_q$-robust for any $q\in[0,1)$. In particular, it is not $\psi_0$-robust, i.e.\ not Hampel robust in the classical sense. For the Average Value at Risk functional evaluated at $\widehat m_n$, i.e.\ for the ordinary empirical Average Value at Risk, the missing Hampel robustness was already observed in \cite{Contetal2010}.

\begin{examplenorm}\label{Example AVaR}
The Average Value at Risk (Expected Shortfall) at level $\alpha\in(0,1)$ is one of the most popular risk measures in finance and insurance. It is defined by $\avatr_\alpha(\cdot):=\frac{1}{\alpha}\int_0^\alpha\vatr_s(\cdot)ds$ on the usual $L^1$-space for any atomless probability space, where $\vatr_s(X)$ refers to the Value at Risk of $X$ at level $s$, i.e.\ to the lower $(1-s)$-quantile of the distribution of $X$. Since $\avatr_\alpha(X)$ depends only on the distribution of $X$, we may identify $\avatr_\alpha$ with a statistical functional ${\cal R}_\alpha$ on ${\cal M}_1^{\psi_1}(\R)$ through ${\cal R}_\alpha(\mu):=\avatr_\alpha(X_\mu)$, where $X_\mu\in L^1$ is any random variable with distribution $\mu$. %Note that ${\cal M}_1^{\psi_1}(\R)$ is just the set of the distributions of all random variables in $L^1$.
It is well known that ${\cal R}_\alpha$ admits the representation
$$%\begin{equation}\label{Example AVaR - Eq - 1}
    {\cal R}_\alpha(\mu)\,=\,-\int_{-\infty}^0g_\alpha(F_\mu(x))\,dx+\int_0^\infty(1-g_\alpha(F_\mu(x)))\,dx
$$%\end{equation}
with $g_\alpha(s):=\frac{1}{\alpha}\max\{s-(1-\alpha);0\}$, where $F_\mu$ stands for the distribution function of $\mu$.

For any $\mu\in{\cal M}_1^{\psi_1}(\R)$ the quantity
$$
    T_n(\mu):={\cal R}_\alpha(\mu^{*n})/n
$$
can be seen as a suitable insurance premium for each individual risk of an insurance collective consisting of $n$ risks that are i.i.d.\ according to $\mu$, where $\mu^{*n}$ refers to the $n$-fold convolution of $\mu$. The individual premium $T_n(\mu)$ depends on the size $n$ of the collective, which reflects the balancing of risks in collectives. If the claims $X_1,\ldots,X_n$ of the $n$ individual risks could be observed in the previous insurance period, then
\begin{equation}\label{Example AVaR - Eq - 2}
    \widehat T_n:=T_n(\widehat m_n)={\cal R}_\alpha(\widehat m_n^{*n})/n
\end{equation}
with $\widehat m_n:=\frac{1}{n}\sum_{i=1}^n\delta_{X_i}$ provides a reasonable estimator for the individual premium $T_n(\mu)$ for the next insurance period. This setting meets the framework above for any $\Theta\subset{\cal M}_1^{\psi_1}(\R)$. For background on the estimation of premiums for individual risks of insurance collectives, see \cite{LauerZaehle2015}. The following Lemma \ref{Example AVaR - Lemma} shows that on any $\Theta$ nested between ${\cal M}_1^\infty(\R)$ and ${\cal M}_1^{\psi_1}(\R)$ (with ${\cal M}_1^\infty(\R)$ as defined below), the sequence $(\widehat T_n)$ is $\psi_1$-robust but not $\psi_q$-robust for any $q\in[0,1)$.
{\hspace*{\fill}$\Diamond$\par\bigskip}
\end{examplenorm}

Let ${\cal M}_1^\infty(\R)$ be the set of all Borel probability measures on $\R$ with compact support.

\begin{lemma}\label{Example AVaR - Lemma}
The sequence $(\widehat T_n)$ introduced in (\ref{Example AVaR - Eq - 2}) in Example \ref{Example AVaR} is $\psi_1$-robust on ${\cal M}_1^{\psi_1}(\R)$ but not $\psi_q$-robust on ${\cal M}_1^\infty(\R)$ for $q\in[0,1)$.
\end{lemma}

\begin{proof}
Part (i) of Theorem \ref{hampel-huber generalized - plug-in estimator - corollary} ensures that the sequence $(\widehat T_n)$ is $\psi_1$-robust on ${\cal M}_1^{\psi_1}(\R)$, because the family $\{T_n:n\in\N\}$ is equicontinuous on ${\cal M}_1^{\psi_1}(\R)$ w.r.t.\ the $\psi_1$-weak topology. The equicontinuity of the family $\{T_n:n\in\N\}$ can be seen as follows. On the one hand, the $\psi_1$-weak topology is metrizable by the $L^1$-Wasserstein metric $d_{\mbox{\scriptsize{\rm W,1}}}(\mu_1,\mu_2):=\int_{-\infty}^\infty|F_{\mu_1}(x)-F_{\mu_1}(x)|\,dx$ (cf.\ Remark 2.9 in \cite{Kraetschmeretal2014}), and on the other hand we have
\begin{eqnarray*}
    |T_n(\mu_1)-T_n(\mu_2)|
    & = & |{\cal R}_\alpha(\mu_1^{*n})/n-{\cal R}_\alpha(\mu_2^{*n})/n| ~\le~\frac{1}{\alpha n}\int_{-\infty}^\infty|F_{\mu_1^{*n}}(x)-F_{\mu_2^{*n}}(x)|\,dx\\
    & = & \frac{1}{\alpha n}\,d_{\mbox{\scriptsize{\rm W,1}}}(\mu_1^{*n},\mu_2^{*n})~\le~\frac{1}{\alpha n}\,n\,d_{\mbox{\scriptsize{\rm W,1}}}(\mu_1,\mu_2)~=~\frac{1}{\alpha}\,d_{\mbox{\scriptsize{\rm W,1}}}(\mu_1,\mu_2).
\end{eqnarray*}
For the second last step we used the inequality $d_{\mbox{\scriptsize{\rm W,1}}}(\mu_1^{*n},\mu_2^{*n})\le n\,d_{\mbox{\scriptsize{\rm W,1}}}(\mu_1,\mu_2)$, which can be easily shown by an induction on $n$.

To prove the second assertion, let $T(\mu):=\int x\,\mu(dx)$. On the one hand, the strong law of large numbers and Theorem 2.4\,(v) (and Example 2.3) in \cite{LauerZaehle2015} show that $\lim_{n\to\infty}\widehat T_n=T(\mu)$ $\pr^\mu$-a.s.\ for every $\mu\in{\cal M}_1^\infty(\R)$. On the other hand, the proof of Lemma \ref{lemma on continuity of pth moment} below shows that the functional $T$ on ${\cal M}_1^\infty(\R)$ is not continuous w.r.t.\ the $\psi_q$-weak topology for $q\in[0,1)$. So $(\widehat T_n)$ cannot be $\psi_q$-robust on ${\cal M}_1^\infty(\R)$ for $q\in[0,1)$, because otherwise we would obtain a contradiction to part (ii) of Theorem \ref{hampel-huber generalized - plug-in estimator - corollary}.
\end{proof}

%%%%%%%%%%%%%%%%%%%%%%%%%%%%%%%%%%%%%%%%%%%%%%%%%%%%%%%%%%%%%%%%
%%%%%%%%%%%%%%%%%%%%%%%%%%%%%%%%%%%%%%%%%%%%%%%%%%%%%%%%%%%%%%%%
%%%%%%%%%%%%%%%%%%%%%%%%%%%%%%%%%%%%%%%%%%%%%%%%%%%%%%%%%%%%%%%%

\subsubsection{A degree of robustness}\label{sec examples - Degree}

In Example \ref{Example AVaR} we have seen that the sequence $(\widehat T_n)$ introduced in (\ref{Example AVaR - Eq - 2}) is $\psi_1$-robust but not $\psi_q$-robust for any $q\in[0,1)$. To some extent, this is a statement about the ``degree'' of robustness of the sequence $(\widehat T_n)$. In fact, the ``degree'' of robustness can be formalized as follows. Let $E=\R$ and $\Theta\subset{\cal M}_1(\R)$, and assume as before that $\mathfrak{E}_n\subset\Theta$ for all $n\in\N$. Then, for every sequence $(T_n)$ of maps $T_n:\Theta\to\Sigma$ we may define the index of robustness for the sequence $(\widehat T_n)\equiv(T_n(\widehat m_n))$ by
$$%\begin{equation}\label{Def IOR}
    {\rm ior}_\Theta(T_n)\,:=\,1/\inf\big\{p\in[0,\infty):\,(\widehat T_n)\mbox{ is $\psi_p$-robust on }\Theta\big\}
$$%\end{equation}
with the conventions $\inf\emptyset:=\infty$, $1/\infty:=0$, and $1/0:=\infty$. On a more informal level amd in a slightly different setting, this concept was already proposed in Remark 3.8 in \cite{Kraetschmeretal2012}; see also Section 2.4 in \cite{Kraetschmeretal2014} for a special case. Of course, in the same way we can define the index of robustness in the case $E=\R^d$; just replace the right-hand side in (\ref{def of psi p}) by $(1+\|x\|)^p$.

We have already seen in the proof of Lemma \ref{Example AVaR - Lemma} that Theorem \ref{hampel-huber generalized - plug-in estimator - corollary} is a useful tool for the specification of the index of robustness. We will now consider another example, namely the sample $p$-th absolute moment. On the one hand, this is more or less a toy example. On the other hand, Corollary \ref{robustness of pth moment} shows that the range of ${\rm ior}$ is a continuum and that the concept of the ${\rm ior}$ is compatible with our intuition that the degree of robustness of the sample $p$-th absolute moment is the higher the smaller $p$ is. For any $p\in(0,\infty)$ let
$$
    T^{(p)}(\mu):=\int |x|^p\,\mu(dx),\quad \mu\in{\cal M}_1^{\psi_p}(\R),
$$
and
$$
    \widehat T_n^{(p)}(x)\,:=\,T^{(p)}(\widehat m_n(x_1,\ldots,x_n))\,=\,\frac{1}{n}\sum_{i=1}^n|x_i|^p,\quad x=(x_1,x_2,\ldots)\in\R^\N,
$$
and recall that ${\cal M}_1^\infty(\R)$ refers to the set of all Borel probability measures on $\R$ with compact support.

\begin{lemma}\label{lemma on continuity of pth moment}
For any $p\in(0,\infty)$, the functional $T^{(p)}$ is $\psi_p$-weakly continuous on ${\cal M}_1^{\psi_p}(\R)$ but not $\psi_q$-weakly continuous on ${\cal M}_1^\infty(\R)$ for $q\in[0,p)$.
\end{lemma}

\begin{proof}
The functional $T^{(p)}$ is clearly $\psi_p$-weakly continuous on ${\cal M}_1^{\psi_p}(\R)$. To verify the second claim we have to show that for every $q\in[0,p)$ there are some $\mu,\mu_1,\mu_2,\ldots\in{\cal M}_1^\infty(\R)$ such that $\mu_n\to\mu$ $\psi_q$-weakly but $T^{(p)}(\mu_n)\not\to T^{(p)}(\mu)$. By the Representation theorem 3.5 in \cite{Kraetschmeretal2014}, it suffices to show that for some atomless probability space $(\overline{\Omega},\overline{{\cal F}},\overline{\pr})$ there exists $X,X_1,X_2,\ldots\in L^\infty(\overline{\Omega},\overline{{\cal F}},\overline{\pr})$ such that $\|X_n-X\|_q\to 0$ for every $q\in(0,p)$ but $\overline{\ex}[|X_n|^p]\not\to\overline{\ex}[|X|^p]$. Let $(\overline{\Omega},\overline{{\cal F}},\overline{\pr})=([0,1],{\cal B}([0,1]),\ell|_{[0,1]})$ with $\ell$ the Lebesgue measure on the line. Set $X:\equiv 0$ and $X_n(x):=w_n\eins_{[1/n,1]}(x)x^{-1/p}$, $x\in[0,1]$, $n\in\N$, for any sequence $(w_n)\subset\R_+$ such that $w_n\to 0$ and $w_n^p\,\log n\nearrow\infty$. Then $\|X_n\|_q\to 0$ for every $q\in(0,p)$ but $\overline{\ex}[|X_n|^p]\nearrow\infty$.
\end{proof}

\begin{corollary}\label{robustness of pth moment}
Let $p\in(0,\infty)$, and $\Theta$ be any set of Borel probability measures on $\R$ such that ${\cal M}_1^\infty(\R)\subset\Theta\subset{\cal M}_1^{\psi_p}(\R)$. Then on $\Theta$ the sequence $(\widehat T_n^{(p)})$ is $\psi_p$-robust but not $\psi_q$-robust for $q\in[0,p)$. In particular, ${\rm ior}_\Theta(T^{(p)})=1/p$.
\end{corollary}

\begin{proof}
By Lemma \ref{lemma on continuity of pth moment}, the functional $T^{(p)}$ is $\psi_p$-weakly continuous on ${\cal M}_1^{\psi_p}(\R)$ but not $\psi_q$-weakly continuous on ${\cal M}_1^\infty(\R)$ for $q\in[0,p)$. Moreover, the sample $p$-th absolute moment is known to be weakly consistent on ${\cal M}_1^{\psi_p}(\R)$ in the sense that $\lim_{n\to\infty}\pr^\mu[|\widehat T_n^{(p)}-T^{(p)}(\mu)|\ge\eta\big]=0$ for all $\eta>0$ and $\mu\in{\cal M}_1^{\psi_p}(\R)$. Thus, the claim of Corollary \ref{robustness of pth moment} is an immediate consequence of Theorem \ref{hampel-huber generalized - plug-in estimator - corollary}.
\end{proof}

%%%%%%%%%%%%%%%%%%%%%%%%%%%%%%%%%%%%%%%%%%%%%%%%%%%%%%%%%%%%%%%%
%%%%%%%%%%%%%%%%%%%%%%%%%%%%%%%%%%%%%%%%%%%%%%%%%%%%%%%%%%%%%%%%
%%%%%%%%%%%%%%%%%%%%%%%%%%%%%%%%%%%%%%%%%%%%%%%%%%%%%%%%%%%%%%%%
%%%%%%%%%%%%%%%%%%%%%%%%%%%%%%%%%%%%%%%%%%%%%%%%%%%%%%%%%%%%%%%%
%%%%%%%%%%%%%%%%%%%%%%%%%%%%%%%%%%%%%%%%%%%%%%%%%%%%%%%%%%%%%%%%
%%%%%%%%%%%%%%%%%%%%%%%%%%%%%%%%%%%%%%%%%%%%%%%%%%%%%%%%%%%%%%%%

\subsection{Estimators in dominated parametric statistical models}\label{Sec exponential models}

In this section we consider point estimators in parametric statistical models. Although we allow $\Theta$ and $\Sigma$ to be rather general metric spaces, in applications one often has $\Theta\subset\R^k$ and $\Sigma=\R^d$. Theorems \ref{hampel-huber generalized - parametric model - finite sample} and \ref{hampel-huber generalized - parametric model - asymptotically} provide general criteria for respectively finite-sample robustness and asymptotic robustness in dominated statistical models. From these criteria one can easily derive robustness of very classical point estimators; cf.\ Examples \ref{hampel-huber generalized - parametric model - finite sample - remark} and \ref{hampel-huber generalized - parametric model - asymptotically - example}. Throughout we will restrict ourselves to a fixed aspect, that is, $T_n=T$ for all $n\in\N$.

Let $(\Theta,d_\Theta)$ be a metric space, $(\Sigma,d_\Sigma)$ be a complete and separable metric space, and ${\cal S}$ be the corresponding Borel $\sigma$-field on $\Sigma$. Let $(E,{\cal E})$ be a measurable space, $(\Omega,{\cal F}):=(E^\N,{\cal E}^{\otimes\N})$, $X_i$ be the $i$-th coordinate projection on $\Omega$, and $\pr^\theta$ be any probability measure on $(\Omega,{\cal F})$ for every $\theta\in\Theta$. We will assume that our parametric statistical model is dominated. That is, we assume that for every $n\in\N$ there is some $\sigma$-finite measure $\mu_n$ on $(E^n,{\cal E}^{\otimes n})$ (called the dominating measure) such that for every $\theta\in\Theta$ the law $\pr^\theta\circ(X_1,\ldots,X_n)^{-1}$ is absolutely continuous w.r.t.\ $\mu_n$ with Radon--Nikodym derivative
$$
    L_n(\,\cdot\,;\theta)\,:=\,\frac{d(\pr^\theta\circ(X_1,\ldots,X_n)^{-1})}{d\mu_n}\,.
$$
Let $T:\Theta\to\Sigma$ be any function. For every $n\in\N$, let $\widehat T_n:\Omega\to\R$ be any map such that $\widehat T_n(x)=\widehat T_n(x_1,\ldots,x_n)$ for all $x=(x_1,x_2,\dots)\in\Omega$, and assume that $\widehat T_n$ is $({\cal F},{\cal S})$-measurable.

\begin{theorem}\label{hampel-huber generalized - parametric model - finite sample}
Take the notation from above and let $\Theta_0\subset\Theta$. Assume that $(E,d_E)$ is a complete and separable metric space and that ${\cal E}$ is the corresponding Borel $\sigma$-field. Let $d_{E^n}$ be defined as in (\ref{def metric on En}), and assume that the following two conditions hold:
\begin{itemize}
    \item[(i)] $x^n\mapsto\widehat T_n(x^n)$ is $(d_{E^n},d_\Sigma)$-continuous on $E^n$ for every $n\in\N$.
    \item[(ii)] For every $\theta_1\in\Theta_0$, $\varepsilon>0$, and $n\in\N$ there is some $\delta>0$ such that
    $$
        \theta_2\in\Theta_0,\quad d_\Theta(\theta_1,\theta_2)\le\delta\quad\Longrightarrow\quad \int |L_n(x^n;\theta_1)-L_n(x^n;\theta_2)|\,\mu_n(dx^n)\,\le\,\varepsilon.
    $$
\end{itemize}
Then the sequence $(\widehat T_n)$ is finite-sample $(d_\Theta,\rho_{\mbox{\scriptsize{\rm P}}})$-robust on $\Theta_0$.
\end{theorem}

\begin{proof}
It suffices to check that conditions (c)--(d) in Theorem \ref{hampel-huber generalized - finite sample} hold. Condition (c) is nothing but assumption (i). So it suffices to show that condition (d) holds, that is, that the mapping
$$
    \theta\mapsto\pr^\theta\circ(X_1,\ldots,X_n)^{-1}[\,\cdot\,]\,=\,\int_{\cdot}L_n(x^n;\theta)\,\mu_n(dx^n)
$$
is $(d_\Theta,\rho_{n})$-continuous on $\Theta_0$ for every $n\in\N$. Without loss of generality we may assume that $\rho_{n}$ is the Prohorov metric (defined analogously to (\ref{def p metric})) on ${\cal M}_1(E^n)$. Let $\theta_1\in\Theta_0$, $\varepsilon>0$, and $n\in\N$ be fixed. By assumption (ii) we can find some $\delta>0$ such that for every $\theta_2\in\Theta_0$ with $d(\theta_1,\theta_2)\le\delta$ and $A\in{\cal E}^{\otimes n}$,
\begin{eqnarray*}
    \lefteqn{\pr^{\theta_1}\circ(X_1,\ldots,X_n)^{-1}[A]}\\
    & = & \int_A L_n(x^n;\theta_1)\,\mu_n(dx^n)\\
    & \le & \int_A L_n(x^n;\theta_2)\,\mu_n(dx^n)\,+\,\int|L_n(x^n;\theta_1)-L_n(x^n;\theta_2)|\,\mu_n(dx^n)\\
    & \le & \pr^{\theta_2}\circ(X_1,\ldots,X_n)^{-1}[A]\,+\,\varepsilon\\
    & \le & \pr^{\theta_2}\circ(X_1,\ldots,X_n)^{-1}[A^\varepsilon]\,+\,\varepsilon.
\end{eqnarray*}
That is, $\rho_{n}(\pr^{\theta_1}\circ(X_1,\ldots,X_n)^{-1},\pr^{\theta_2}\circ(X_1,\ldots,X_n)^{-1})\le\varepsilon$. This completes the proof.
\end{proof}

\begin{examplenorm}\label{hampel-huber generalized - parametric model - finite sample - remark}
Conditions (i)--(ii) in Theorem \ref{hampel-huber generalized - parametric model - finite sample} are easily checked in many specific situations. If $\pr^\theta$ is an infinite product measure of the shape $\pr^\theta=(\pr_1^\theta)^{\otimes\N}$ with $\pr_1^\theta[\,\cdot\,]:=\int_{\cdot}L_1(x^1;\theta)\,\mu_1(dx^1)$ (corresponding to i.i.d.\ observations drawn from $\pr_1^\theta$) for every $\theta\in\Theta$, then condition (ii) reduces to the condition (ii)' that for every $\theta_1\in\Theta_0$ and $\varepsilon>0$ there be some $\delta>0$ such that
\begin{equation}\label{hampel-huber generalized - parametric model - finite sample - remark - 10}
    \theta_2\in\Theta_0,\quad d_\Theta(\theta_1,\theta_2)\le\delta\quad\Longrightarrow\quad \int|L_1(x^1;\theta_1)-L_n(x^1;\theta_2)|\,\mu_1(dx^1)\,\le\,\varepsilon.
\end{equation}
If, for example, $\pr_1^\theta$ is the Bernoulli distribution ${\rm B}_{1,\theta}$, the Poisson distribution ${\rm Poiss}_\theta$, the exponential distribution ${\rm Exp}_\theta$, or the normal distribution ${\rm N}_{\theta,\sigma^2}$, then condition (ii)' is easily seen to hold for respectively $\Theta_0=(0,1)$, $\Theta_0=(0,\infty)$, $\Theta_0=(0,\infty)$, $\Theta_0=\R$, and condition (i) is easily seen to hold for the corresponding well known maximum likelihood estimators for $\theta$.
{\hspace*{\fill}$\Diamond$\par\bigskip}
\end{examplenorm}

Let us now assume that $\Sigma=\R$, and that the estimator $\widehat T_n$ is unbiased for $T(\theta)$ for every $n\in\N$. In this case we know from Theorem \ref{hampel-huber generalized} and Remark \ref{hampel-huber generalized - remark on variance} that the sequence $(\widehat T_n)$ is asymptotically $(d_\Theta,\rho_{\mbox{\scriptsize{\rm P}}})$-robust on $\Theta_0\subset\Theta$ if $T|_{\Theta_0}$ is continuous and if for every $\theta_1\in\Theta_0$ and $\varepsilon>0$ there are some $\delta>0$ and $n_0\in\N$ such that $\vari^{\theta_2}[\widehat T_n]\le\varepsilon$ holds for all $n\ge n_0$ and $\theta_2\in\Theta_0$ satisfying $d_\Theta(\theta_1,\theta_2)\le\delta$. If $\Theta$ is an open interval in $\R$, then under conditions (i)--(vi) in the following Theorem \ref{hampel-huber generalized - parametric model - asymptotically} (in this case we speak of a ``regular'' dominated statistical model) the commonly known Cramér--Rao--Fréchet information inequality
\begin{equation}\label{information inequality}
    \vari^\theta\big[\widehat T_n\big]\,\ge\,T'(\theta)^2/I_n(\theta)
\end{equation}
holds for every $\theta\in\Theta$ and $n\in\N$, where the Fisher information $I_n$ is defined in (\ref{def fisher information}) below. The estimator $\widehat T_n$ is said to be Cramér--Rao efficient for $T(\theta)$ if equality holds in (\ref{information inequality}).

\begin{theorem}\label{hampel-huber generalized - parametric model - asymptotically}
Take the notation from above, let $\Theta$ be an open interval in $\R$ (possibly infinite or semi-infinite) and $\Theta_0\subset\Theta$, and assume that the following eight conditions hold:
\begin{itemize}
    \item[(i)] For every $x^n\in E^n$ and $n\in\N$, the mapping $\theta\mapsto L_n(x^n;\theta)$ is strictly positive and continuously differentiable on $\Theta$.
    \item[(ii)] For every $\theta\in\Theta$ and $n\in\N$, we have that
    $$
        \int \frac{\partial}{\partial\theta} L_n(x^n;\theta)\,\mu_n(dx^n)\,=\,\frac{\partial}{\partial\theta}\int L_n(x^n;\theta)\,\mu_n(dx^n).
    $$
    \item[(iii)] For every $\theta\in\Theta$ and $n\in\N$, the Fisher information
    \begin{equation}\label{def fisher information}
        I_n(\theta)\,:=\,\int \Big(\frac{\partial}{\partial\theta}\log L_n(x^n;\theta)\Big)^2L_n(x^n;\theta)\,\mu_n(dx^n)
    \end{equation}
    is finite and strictly positive.
    \item[(iv)] $T$ is continuously differentiable with $T'(\theta)\not=0$ for all $\theta\in\Theta$.
    \item[(v)] For every $\theta\in\Theta$ and $n\in\N$, we have that $\ex^\theta[\widehat T_n]=T(\theta)$.
    \item[(vi)] For every $\theta\in\Theta$ and $n\in\N$, we have that
    $$
        \int \widehat T_n(x^n)\frac{\partial}{\partial\theta} L_n(x^n;\theta)\,\mu_n(dx^n)\,=\,\frac{\partial}{\partial\theta}\int \widehat T_n(x^n)L_n(x^n;\theta)\,\mu_n(dx^n).
    $$
    \item[(vii)] $\widehat T_n$ is Cramér--Rao efficient for $T(\theta)$ for every $n\in\N$.
    \item[(viii)] For every $\theta_1\in\Theta_0$ there is some $\delta>0$ such that
    $$
        \sup_{\theta\in\Theta_0:\,d_\Theta(\theta_1,\theta)\le\delta}T'(\theta)^2/I_1(\theta)\,<\,\infty.
    $$
\end{itemize}
Then the sequence $(\widehat T_n)$ is asymptotically $(|\cdot|,\rho_{\mbox{\scriptsize{\rm P}}})$-robust on $\Theta_0$.
\end{theorem}

\begin{proof}
It suffices to check that conditions (a)--(b) in Theorem \ref{hampel-huber generalized} hold. Conditions (a) in Theorem \ref{hampel-huber generalized} holds by assumptions (iv). Assumption (vii) (along with (i)--(vi)) ensures that equality holds in (\ref{information inequality}). Together with Remark \ref{hampel-huber generalized - remark on variance}, the well known identity $I_n(\cdot)=n I_1(\cdot)$, and assumption (viii), this ensures that assumption (b) in Theorem \ref{hampel-huber generalized} holds too.
\end{proof}

\begin{remarknorm}\label{hampel-huber generalized - parametric model - asymptotically - remark}
The considered dominated statistical model is said to be exponential when for every $n\in\N$ the Likelihood function $L_n$ has the shape
\begin{equation}\label{exponential likelihood}
    L_n(x^n;\theta)\,=\,h_n(x^n)\,c_n(\theta)\,e^{q_n(\theta)\widehat T_n(x^n)}\quad\mbox{ for all $x^n\in E^n$, $\theta\in\Theta$}
\end{equation}
for an $({\cal E}^{\otimes n},{\cal B}(\R))$-measurable function $h_n:E^n\to\R$ and two functions $c_n,q_n:\Theta\to\R$. If an exponential statistical model satisfies conditions (i)--(vi) in Theorem \ref{hampel-huber generalized - parametric model - asymptotically}, and $q_n(\Theta)$ is open and $q_n$ is continuously differentiable on $\Theta$ for every $n\in\N$, then the sequence $(\widehat T_n)$ satisfies condition (vii) in Theorem \ref{hampel-huber generalized - parametric model - asymptotically}; see, for instance, \cite{Wijsman1973}. Wijsman \cite{Wijsman1973} even showed that under conditions (i)--(vi) the representation (\ref{exponential likelihood}) (with $q_n$ continuously differentiable on $\Theta$) is not only sufficient but also necessary for condition (vii).
{\hspace*{\fill}$\Diamond$\par\bigskip}
\end{remarknorm}

\begin{examplenorm}\label{hampel-huber generalized - parametric model - asymptotically - example}
For instance, if $\pr_1^\theta$ is the Bernoulli distribution ${\rm B}_{1,\theta}$, the Poisson distribution ${\rm Poiss}_\theta$, the exponential distribution ${\rm Exp}_\theta$, or the normal distribution ${\rm N}_{\theta,\sigma^2}$, then the conditions of Theorem \ref{hampel-huber generalized - parametric model - asymptotically} are fulfilled for the corresponding maximum likelihood estimators for $\theta$ and for $\Theta_0=(0,1)$, $\Theta_0=(0,\infty)$, $\Theta_0=(0,\infty)$, and $\Theta_0=\R$, respectively. The validity of (vii) follows from Remark \ref{hampel-huber generalized - parametric model - asymptotically - remark} and the fact that each of these distributions induces an exponential statistical model (subject to conditions (i)--(vi)) for their maximum likelihood estimators. For (viii) note that the Fisher informations of these distributions are $I_n(\theta)=1/(\theta(1-\theta))$, $I_1(\theta)=1/\theta$, $I_1(\theta)=1/\theta^2$, and $I_n\equiv1/\sigma^2$, respectively.
{\hspace*{\fill}$\Diamond$\par\bigskip}
\end{examplenorm}

%%%%%%%%%%%%%%%%%%%%%%%%%%%%%%%%%%%%%%%%%%%%%%%%%%%%%%%%%%%%%%%%
%%%%%%%%%%%%%%%%%%%%%%%%%%%%%%%%%%%%%%%%%%%%%%%%%%%%%%%%%%%%%%%%
%%%%%%%%%%%%%%%%%%%%%%%%%%%%%%%%%%%%%%%%%%%%%%%%%%%%%%%%%%%%%%%%
%%%%%%%%%%%%%%%%%%%%%%%%%%%%%%%%%%%%%%%%%%%%%%%%%%%%%%%%%%%%%%%%
%%%%%%%%%%%%%%%%%%%%%%%%%%%%%%%%%%%%%%%%%%%%%%%%%%%%%%%%%%%%%%%%
%%%%%%%%%%%%%%%%%%%%%%%%%%%%%%%%%%%%%%%%%%%%%%%%%%%%%%%%%%%%%%%%

\subsection{Yule--Walker estimator for the parameter of a linear process}\label{Sec autoregressive process}

In this section we consider an example for a statistical model whose parameter space $\Theta$ is the product set of an interval in the real line and a ``nonparametric'' set of Borel probability measures on the real line. To some extent, this model can be referred to as semiparametric. To describe the model, assume that we can observe a linear process
\begin{equation}\label{def ar 1 process}
    X_n\,=\,\sum_{k=0}^\infty a^kZ_{n-k},\quad n\in\N
\end{equation}
for any $a\in(-1,1)$ and any sequence $(Z_k)_{k\in\Z}$ of i.i.d.\ real-valued random variables with expectation zero and a finite second moment distinct from zero, but we do not know the constant $a$ (and the distribution $\mu$ of $Z_1$). The assumption $|a|<1$ ensures that $(X_n)$ is a strictly stationary, zero mean $L^2$-process. In view of
$$
    \frac{\ex^{(a,\mu)}[X_1X_2]}{\ex^{(a,\mu)}[X_1^2]}\,=\,a\qquad\mbox{(if $a\not=0$)}
$$
a reasonable estimators for $a$ based on the first $n$ observations is given by $\widehat T_n(X_1,\ldots,X_n)$, where
\begin{equation}\label{def yule walker}
    \widehat T_n(x_1,\ldots,x_n)\,:=\,
        \left\{
        \begin{array}{lll}
            \frac{\frac{1}{n-1}\sum_{i=1}^{n-1}x_i\,x_{i+1}}{\frac{1}{n}\sum_{i=1}^{n}x_i^2} & , & \sum_{i=1}^{n}x_i^2>0\\
            0 & , & \mbox{else}
        \end{array}
    \right..
\end{equation}
This situation corresponds to a statistical model parameterized by the couple $(a,\mu)$, where the interest is in the aspect $T(a,\mu)=a$. Theorem \ref{hampel-huber generalized - Yule Walker} below shows that a small change in $a$ only leads to a small change in the distribution of the estimator $\widehat T_n$ uniformly in $n$. Since a small change in $a$ might come along with a small change in $\mu$, we also allow for small changes in $\mu$. However, one can also assume that $\mu$ is fixed if one starts from the premise that the distribution of the noise cannot change. In this case the change in the underlying model can be measured simply by the Euclidean distance $|a_1-a_2|$ and the semiparametric model turns into a parametric model with parameter space $(-1,1)$.

\begin{remarknorm}
When $a$ and $(Z_k)_{k\in\N}$ are as above, then the AR(1) process $X_n=a X_{n-1}+Z_n$, $n\in\N$, has the representation (\ref{def ar 1 process}); cf.\ Example 3.1.2 (or Theorem 3.1.1) in \cite{BrockwellDavis2006}. In this case, the estimator $\widehat T_n$ defined in (\ref{def yule walker}) is also called Yule--Walker estimator for $a$.
{\hspace*{\fill}$\Diamond$\par\bigskip}
\end{remarknorm}

Let $\Theta:=(-1,1)\times{\cal M}_1^{0,2}(\R)$, where ${\cal M}_1^{0,2}(\R)$ is the set of all probability measures on $\R$ with mean zero and a finite second moment distinct from zero. Choose the metric
$$
    d_\Theta((a_1,\mu_1),(a_2,\mu_2))\,:=\,\max\{|a_1-a_2|\,;\,d_{0,2}(\mu_1,\mu_2)\}
$$
on $\Theta$, where $d_{0,2}$ is any metric which metrizes the weak topology on ${\cal M}_1^{0,2}(\R)$. Let $(\Sigma,d_\Sigma)=((-1,1),|\cdot|)$, ${\cal S}={\cal B}((-1,1))$, and $T:\Theta\to(-1,1)$ be defined by $T(a,\mu):=a$.
Let $(\Omega,{\cal F}):=(\R^\N,{\cal B}(\R)^{\otimes\N})$, $X_i$ be the $i$-th coordinate projection on $\Omega$, and $\pr^{(a,\mu)}$ be the law on $(\Omega,{\cal F})$ of the linear process defined in (\ref{def ar 1 process}) when $Z_1$ is distributed according to $\mu$, $(a,\mu)\in\Theta$. Finally, for any $\Theta_0\subset\Theta$ let
$$
    {\cal M}_0\,:=\,\{\mu:(a,\mu)\in\Theta_0\mbox{ for some }a\in(-1,1)\}
$$
and recall from Definition \ref{def of uniformly psi integrating} the meaning of ``locally uniformly $\psi$-integrating''.

\begin{theorem}\label{hampel-huber generalized - Yule Walker}
Take the notation from above, let $\Theta_0\subset\Theta$, and assume that the set ${\cal M}_0$ is locally uniformly $\psi_2$-integrating with $\psi_2$ as in (\ref{def of psi p}) and that all elements of ${\cal M}_0$ are absolutely continuous w.r.t.\ the Lebesgue measure. Then the sequence $(\widehat T_n)$ is $(d_\Theta,\rho_{\mbox{\scriptsize{\rm P}}})$-robust on $\Theta_0$.
\end{theorem}

The proof of Theorem \ref{hampel-huber generalized - Yule Walker} relies on the following two lemmas.

\begin{lemma}\label{hampel-huber generalized - Yule Walker - Lemma}
Take the notation from above, let $\Theta_0\subset\Theta$, and assume that the set ${\cal M}_0$ is locally uniformly $\psi_2$-integrating. Let $\ell\in\N_0$. Then for every $(a_1,\mu_1)\in\Theta_0$, $\varepsilon>0$, and $\eta>0$ there exist some $\delta>0$ and $n_0\in\N$ such that for every $n\ge n_0$ and $(a_2,\mu_2)\in\Theta_0$ with $d_\Theta((a_1,\mu_1),(a_2,\mu_2))\le\delta$ we have that
\begin{equation}\label{hampel-huber generalized - Yule Walker - Lemma - EQ}
    \pr^{(a_2,\mu_2)}\Big[\Big|\frac{1}{n}\sum_{i=1}^nX_{i}X_{i+\ell}-\ex^{(a_2,\mu_2)}[X_1X_{1+\ell}]\Big|\ge\eta\Big]\,\le\,\varepsilon.
\end{equation}
When the set ${\cal M}_0$ is even uniformly $\psi_2$-integrating, then for every $\alpha\in(0,1)$, $\varepsilon>0$, and $\eta>0$ there exists some $n_0\in\N$ such that the inequality (\ref{hampel-huber generalized - Yule Walker - Lemma - EQ}) holds for all $n\ge n_0$ and $(a_2,\mu_2)\in\Theta_0^\alpha\,:=\{(a,\mu)\in\Theta_0: |a|\le\alpha\}$.
\end{lemma}

\begin{proof}
Let $(\overline{\Omega},\overline{{\cal F}}):=(\R^\Z,{\cal B}(\R)^{\otimes\Z})$, $Z_k$ be the $k$-th coordinate projection on $\overline{\Omega}$, and $\overline{\pr}^{\mu}:=\mu^{\otimes\Z}$, $\mu\in{\cal M}_1^{0,2}(\R)$. Then, for every $(a,\mu)\in\Theta$, we have
\begin{eqnarray}
    \lefteqn{\pr^{(a,\mu)}\Big[\Big|\frac{1}{n}\sum_{i=1}^nX_{i}X_{i+\ell}-\ex^{(a,\mu)}[X_1X_{1+\ell}]\Big|\ge\eta\Big]}\nonumber\\
    & = & \overline{\pr}^\mu\Big[\Big|\frac{1}{n}\sum_{i=1}^n\sum_{k=0}^\infty\sum_{m=0}^\infty a^ka^mZ_{i-k}Z_{i+\ell-m}-\sum_{k=0}^\infty\sum_{m=0}^\infty a^ka^m\overline{\ex}^{\mu_2}[Z_{-k} Z_{\ell-m}]\Big|\ge\eta\Big]\nonumber\\
    & = & \overline{\pr}^\mu\Big[\Big|\sum_{k=0}^\infty\sum_{m=0}^\infty a^{k+m}\,\frac{1}{n}\sum_{i=1}^n\big(Z_{i-k}Z_{i+\ell-m}-\overline{\ex}^{\mu_2}[Z_{-k}Z_{\ell-m}]\big)\Big|\ge\eta\Big]\nonumber\\
    & =: & P(n,a,\mu,\eta).\nonumber
\end{eqnarray}
Let $(a_1,\mu_1)\in\Theta_0$, $\varepsilon>0$, and $\eta>0$ be fixed. In the following we will show that there exist $\delta>0$ and $n_0\in\N$ such that for all $n\ge n_0$ and $(a_2,\mu_2)\in\Theta_0$ with $d_\Theta((a_1,\mu_1),(a_2,\mu_2))\le\delta$ we have that $P(n;a_2,\mu_2;\eta)\le\varepsilon$.

For every $(a_2,\mu_2)\in\Theta_0$ (with $a_2\not=0$) and $q\in(0,1)$ we have
\begin{eqnarray}
    \lefteqn{P(n,a_2,\mu_2,\eta)}\nonumber\\
    & \le & \overline{\pr}^{\mu_2}\Big[\sum_{k=0}^\infty\sum_{m=0}^\infty |a_2|^{k+m}\Big|\frac{1}{n}\sum_{i=1}^n\big(Z_{i-k}Z_{i+\ell-m}-\overline{\ex}^{\mu_2}[Z_{-k}Z_{\ell-m}]\big)\Big|\ge\eta\Big]\nonumber\\
    & \le & \sum_{k=0}^\infty\sum_{m=0}^\infty\overline{\pr}^{\mu_2}\Big[|a_2|^{k+m}\Big|\frac{1}{n}\sum_{i=1}^n\big(Z_{i-k}Z_{i+\ell-m}-\overline{\ex}^{\mu_2}[Z_{-k}Z_{\ell-m}]\big)\Big|\ge \eta\,q^{k+m}(1-q)^2\Big]\nonumber\\
   & = & \sum_{k=0}^\infty\sum_{m=0}^\infty\overline{\pr}^{\mu_2}\Big[\Big|\frac{1}{n}\sum_{i=1}^n\big(Z_{i-k}Z_{i+\ell-m}-\overline{\ex}^{\mu_2}[Z_{-k}Z_{\ell-m}]\big)\Big|\ge \eta\,|q/a_2|^{k+m}(1-q)^2\Big].\nonumber
\end{eqnarray}
For every $u>0$ and $j\in\Z$, we set $Z_j^u=Z_j\eins_{\{|Z_j|\le u\}}$. Then $Z_jZ_k$ has the representation $Z_jZ_k=Z_j^uZ_k^u+Z_jZ_k\eins_{\{|Z_j|>u\}\cup\{|Z_k|>u\}}$. In particular,
\begin{eqnarray}
    \lefteqn{P(n,a_2,\mu_2,\eta)}\nonumber\\
    & \le & \sum_{k=0}^\infty\sum_{m=0}^\infty\overline{\pr}^{\mu_2}\Big[\Big|\frac{1}{n}\sum_{i=1}^n\big(Z_{i-k}^uZ_{i+\ell-m}^u-\overline{\ex}^{\mu_2}[Z_{-k}^uZ_{\ell-m}^u]\big)\Big|\ge \frac{\eta\,|q/a_2|^{k+m}(1-q)^2}{3}\Big]\nonumber\\
    & & +\,\sum_{k=0}^\infty\sum_{m=0}^\infty\overline{\pr}^{\mu_2}\Big[\Big|\frac{1}{n}\sum_{i=1}^nZ_{i-k}Z_{i+\ell-m}\eins_{\{|Z_{i-k}|>u\}\cup\{|Z_{i+\ell-m}|>u\}}\Big|\ge \frac{\eta\,|q/a_2|^{k+m}(1-q)^2}{3}\Big]\nonumber\\
    & & +\,\sum_{k=0}^\infty\sum_{m=0}^\infty\overline{\pr}^{\mu_2}\Big[\overline{\ex}^{\mu_2}\big[Z_{-k}Z_{\ell-m}\eins_{\{|Z_{-k}|>u\}\cup\{|Z_{\ell-m}|>u\}}\big] \ge \frac{\eta\,|q/a_2|^{k+m}(1-q)^2}{3}\Big]\nonumber\\
    & =: & P_{1}(n,u,q,a_2,\mu_2,\eta)\,+\,P_{2}(n,u,q,a_2,\mu_2,\eta)\,+\,P_{3}(u,q,a_2,\mu_2,\eta).\label{hampel-huber generalized - Yule Walker - PROOF - 20}
\end{eqnarray}
Let $\delta_1\in(0,1-|a_1|)$ and set $q_1:=(1+|a_1|+\delta_1)/2$. Then $q_1\in(|a_1|+\delta_1,1)$. In particular, $|a_2|/q_1\le (|a_1|+\delta_1)/q_1=:q_0\in(0,1)$ when $|a_1-a_2|\le\delta_1$. Since ${\cal M}_0$ is locally uniformly $\psi_2$-integrating, Lemma \ref{weak and psi weak topology} enures that the mapping ${\cal M}_0\ni\mu\mapsto\overline{\ex}^\mu[Z_1^2]$ is $(d_{0,2},|\cdot|)$-continuous, and therefore we can find some $\delta_2\in(0,\delta_1]$ such that $\overline{\ex}^{\mu_2}[Z_1^2]\le \overline{\ex}^{\mu_1}[Z_1^2]+1=:C(\mu_1)$ for all $\mu_2\in\Theta_0$ with $d_{0,2}(\mu_1,\mu_2)\le\delta_2$. Since ${\cal M}_0$ is locally uniformly $\psi_2$-integrating, we can also find some $\delta_3\in(0,\delta_2]$ and $u_0>0$ such that for every $u\ge u_0$ and $\mu_2\in{\cal M}_0$ with $d_{0,2}(\mu_1,\mu_2)\le\delta_3$ we have that $\overline{\ex}^{\mu_2}[Z_{1}^2\eins_{\{|Z_{1}|>u\}}]\le\eta^2(1-q_0)^4(1-q_1)^4 C(\mu_1)^{-1}\,\varepsilon^2/144$. Then, using Markov's inequality and H\"older's inequality, we obtain for every $n\in\N$, $u\ge u_0$, and $(a_2,\mu_2)\in\Theta_0$ with $d_{\Theta}((a,\mu_1),(a_2,\mu_2))\le\delta_3$,
\begin{eqnarray}
    \lefteqn{P_{2}(n,u,q_1,a_2,\mu_2,\eta)}\nonumber\\
    & \le & 3\eta^{-1}(1-q_1)^{-2}\max\big\{\overline{\ex}^{\mu_2}\big[|Z_{0}Z_{1}|\eins_{\{|Z_{0}|>u\}\cup\{|Z_{1}|>u\}}\big]\,;\, \overline{\ex}^{\mu_2}\big[Z_{1}^2\eins_{\{|Z_{1}|>u\}}\big]\big\}\sum_{k=0}^\infty\sum_{m=0}^\infty q_0^{k+m}\nonumber\\
    & = & 3\eta^{-1}(1-q_1)^{-2}\max\big\{\overline{\ex}^{\mu_2}\big[|Z_{0}Z_{1}|\eins_{\{|Z_{0}|>u\}\cup\{|Z_{1}|>u\}}\big]\,;\, \overline{\ex}^{\mu_2}\big[Z_{1}^2\eins_{\{|Z_{1}|>u\}}\big]\big\}\,(1-q_0)^{-2}\nonumber\\
    & \le & 3\eta^{-1}(1-q_1)^{-2}\max\big\{\overline{\ex}^{\mu_2}\big[|Z_{0}Z_{1}|(\eins_{\{|Z_{0}|>u\}}+\eins_{\{|Z_{1}|>u\}})\big]\,;\, \overline{\ex}^{\mu_2}\big[Z_{1}^2\eins_{\{|Z_{1}|>u\}}\big]\big\}\,(1-q_0)^{-2}\nonumber\\
    & \le & 3\eta^{-1}(1-q_1)^{-2}\,2C(\mu_1)^{1/2}\,\overline{\ex}^{\mu_2}\big[|Z_{1}^2|\eins_{\{|Z_{1}|>u\}}\big]^{1/2}\,(1-q_0)^{-2}\nonumber\\
    & \le & \varepsilon/2.\label{hampel-huber generalized - Yule Walker - PROOF - 30}
\end{eqnarray}
By the assumption that ${\cal M}_0$ is locally uniformly $\psi_2$-integrating, and using arguments as for $P_{2}(n,u,q_1,a_2,\mu_2,\eta)$, we can also find some $\delta_4\in(0,\delta_3]$ and $u_1\ge u_0$ such that for every $k,m\in\Z$ and  $\mu_2\in\Theta_0$ with $d_{0,2}(\mu_1,\mu_2)\le\delta_4$ we have $\overline{\ex}^{\mu_2}\big[Z_{-k}Z_{\ell-m}\eins_{\{|Z_{-k}|>u\}\cup\{|Z_{\ell-m}|>u\}}\big]<\eta(1-q_1)^{2}/(3q_0)$. Since $\eta\,|q_1/a_2|^{k+m}(1-q_1)^2/3>\eta(1-q_1)^2/(3q_0)$ when $|a_1-a_2|\le\delta_4$, it follows that for every $(a_2,\mu_2)\in\Theta_0$ with $d_\Theta((a_1,\mu_1),(a_2,\mu_2))\le\delta_4$,
\begin{equation}\label{hampel-huber generalized - Yule Walker - PROOF - 40}
    P_{3}(u_1,q_1,a_2,\mu_2,\eta)\,=\,0.
\end{equation}
Further, let $n_0\in\N$ such that $n_0\ge 108\eta^{-2}(1-q_1)^{-4}(1-q_0^2)^{-2}u_1^2\,\varepsilon^{-1}$. Then, by Markov's inequality we obtain for every $n\ge n_0$ and $(a_2,\mu_2)\in\Theta_0$ with $d_{\Theta}((a,\mu_1),(a_2,\mu_2))\le\delta_4$,
\begin{eqnarray}
    \lefteqn{P_{1}(n,u_1,q_1,a_2,\mu_2,\eta)}\nonumber\\
    & \le & 9\eta^{-2}n^{-2}(1-q_1)^{-4}\sum_{k=0}^\infty\sum_{m=0}^\infty     q_0^{2(k+m)}\,\overline{\ex}^{\mu_2}\Big[\Big(\sum_{i=1}^n\big(Z_{i-k}^{u_1}Z_{i+\ell-m}^{u_1}-\overline{\ex}^{\mu_2}[Z_{-k}^{u_1}Z_{\ell-m}^{u_1}]\big)\Big)^2\Big]\nonumber\\
    & \le & 9\eta^{-2}(1-q_1)^{-4}(1-q_0^2)^{-2}\,n^{-2}\sup_{k,m\in\Z} \,\overline{\ex}^{\mu_2}\Big[\Big(\sum_{i=1}^n\big(Z_{i}^{u_1}Z_{i+k+\ell-m}^{u_1}-\overline{\ex}^{\mu_2}[Z_{0}^{u_1}Z_{k+\ell-m}^{u_1}]\big)\Big)^2\Big]\nonumber\\
    & = & 9\eta^{-2}(1-q_1)^{-4}(1-q_0^2)^{-2}\,n^{-2}\sup_{k,m\in\Z} \,\sum_{i=1}^n\sum_{j=1}^n\overline{\covi}^{\mu_2}\big(Z_{i}^{u_1}Z_{i+k+\ell-m}^{u_1},Z_{j}^{u_1}Z_{j+k+\ell-m}^{u_1}\big)\nonumber\\
    & = & 9\eta^{-2}(1-q_1)^{-4}(1-q_0^2)^{-2}\,n^{-2}\sup_{k,m\in\Z} \,\sum_{i=1}^n\sum_{j\in J(i,k,m)}\overline{\covi}^{\mu_2}\big(Z_{i}^{u_1}Z_{i+k+\ell-m}^{u_1},Z_{j}^{u_1}Z_{j+k+\ell-m}^{u_1}\big)\nonumber\\
    & \le & 9\eta^{-2}(1-q_1)^{-4}(1-q_0^2)^{-2}\,n^{-2}\,(n\cdot3\cdot2u_1^2)\nonumber\\
    & \le & \varepsilon/2\label{hampel-huber generalized - Yule Walker - PROOF - 50}
\end{eqnarray}
with $J(i,k,m):=\{i,i+k+\ell-m,i-k-\ell+m\}$. Altogether, (\ref{hampel-huber generalized - Yule Walker - PROOF - 20})--(\ref{hampel-huber generalized - Yule Walker - PROOF - 50}) imply that for every $n\ge n_0$ and $(a_2,\mu_2)\in\Theta_0$ with $d_\Theta((a_1,\mu_1),(a_2,\mu_2))\le\delta:=\delta_4$ we have that $P(n,a_2,\mu_2,\eta)\le\varepsilon$. This proves the first claim of the lemma. The second claim of the lemma can be shown analogously.
\end{proof}

\begin{lemma}\label{hampel-huber generalized - Yule Walker - Lemma 2}
Take the notation from above, let $\Theta_0\subset\Theta$, and assume that the set ${\cal M}_0$ is locally uniformly $\psi_1$-integrating with $\psi_1$ as in (\ref{def of psi p}). Then, for every $n\in\N$, the mapping
$$
    \Theta_0\ni(a,\mu)\,\longmapsto\,\pr^{(a,\mu)}\circ(X_1,\ldots,X_n)^{-1}
$$
is $(d_\Theta,\rho_{n})$-continuous, where $\rho_{n}$ refers to any metric on ${\cal M}_1(\R^n)$ which metrizes the weak topology.
\end{lemma}

\begin{proof}
Let $(a,\mu)\in\Theta_0$, $(a_m,\mu_m)\subset\Theta_0$, and $n\in\N$. Assume that $d_\Theta((a_m,\mu_m),(a,\mu))\to0$, that is, $a_m\to a$ and $\mu_m$ converges weakly to $\mu$. We have to show that the probability measure $\pr^{(a_m,\mu_m)}\circ(X_1,\ldots,X_n)^{-1}$ converges weakly to $\pr^{(a,\mu)}\circ(X_1,\ldots,X_n)^{-1}$. Let $f:\R^n\to\R$ be a bounded and Lipschitz continuous function; the Lipschitz constant will be denoted by $L$. We have
\begin{eqnarray*}
    \lefteqn{\big|\ex^{(a_m,\mu_m)}[f(X_1,\ldots,X_n)]\,-\,\ex^{(a,\mu)}[f(X_1,\ldots,X_n)]\big|}\\
    & = & \Big|\int f\Big(\sum_{k=0}^\infty a_m^kz_{1-k}\,,\ldots,\sum_{k=0}^\infty a_m^kz_{n-k}\Big)\,\mu_{m}^{\otimes\Z}(d(z_j)_{j\in\Z})\\
    & & \qquad-\,\int f\Big(\sum_{k=0}^\infty a^kz_{1-k}\,,\ldots,\sum_{k=0}^\infty a^kz_{n-k}\Big)\,\mu^{\otimes\Z}(d(z_j)_{j\in\Z})\Big|\\
    & \le & \int \Big|f\Big(\sum_{k=0}^\infty a_m^kz_{1-k}\,,\ldots,\sum_{k=0}^\infty a_m^kz_{n-k}\Big)\\
    & & \qquad-\,f\Big(\sum_{k=0}^\infty a^kz_{1-k}\,,\ldots,\sum_{k=0}^\infty a^kz_{n-k}\Big)\Big|\,\mu_m^{\otimes\Z}(d(z_j)_{j\in\Z})\\
    & & +\,\Big|\int f\Big(\sum_{k=0}^\infty a^kz_{1-k}\,,\ldots,\sum_{k=0}^\infty a^kz_{n-k}\Big)\Big(\mu_m^{\otimes\Z}(d(z_j)_{j\in\Z})-\mu^{\otimes\Z}(d(z_j)_{j\in\Z})\Big)\Big|\\
    & =: & S_1(m)\,+\,S_2(m).
\end{eqnarray*}
On the one hand, using the Lipschitz continuity of $f$ and the Mean value theorem, we obtain
\begin{eqnarray*}
    S_1(m)
    & \le & \int L\sum_{i=1}^n\Big|\sum_{k=1}^\infty (a_m^k-a^k)z_{i-k}\Big|\,\mu_m^{\otimes\Z}(d(z_j)_{j\in\Z})\\
    & \le & L\int\sum_{i=1}^n\sum_{k=1}^\infty k\max\{|a|;|a_m|\}^{k-1}|a_m-a|\,|z_{i-k}|\,\mu_m^{\otimes\Z}(d(z_j)_{j\in\Z})\\
    & \le & \Big(Ln\int |z|\,\mu_m(dz)\sum_{k=1}^\infty k\max\{|a|;|a_m|\}^{k-1}\Big)|a_m-a|.
\end{eqnarray*}
By Lemma \ref{weak and psi weak topology}, the weak convergence of $\mu_m$ to $\mu$ and the assumption on ${\cal M}_0$ imply that  $\int |z|\,\mu_m(dz)\to \int |z|\,\mu(dz)$. Together with $|a|<1$ and $a_m\to a$, this implies $S_1(m)\to 0$.

On the other hand, for any $h\in\N$ we have
\begin{eqnarray*}
    S_2(m)
    & \le & \Big|\int f\Big(\sum_{k=0}^\infty a^kz_{1-k}\,,\ldots,\sum_{k=0}^\infty a^kz_{n-k}\Big)\,\mu_m^{\otimes\Z}(d(z_j)_{j\in\Z})\\
    & & \quad-\int f\Big(\sum_{k=0}^h a^kz_{1-k}\,,\ldots,\sum_{k=0}^h a^kz_{n-k}\Big)\,\mu_m^{\otimes\Z}(d(z_j)_{j\in\Z})\Big|\\
    &  & +\,\Big|\int f\Big(\sum_{k=0}^h a^kz_{1-k}\,,\ldots,\sum_{k=0}^h a^kz_{n-k}\Big)\,\mu_m^{\otimes\Z}(d(z_j)_{j\in\Z})\\
    & & \quad-\int f\Big(\sum_{k=0}^h a^kz_{1-k}\,,\ldots,\sum_{k=0}^h a^kz_{n-k}\Big)\,\mu^{\otimes\Z}(d(z_j)_{j\in\Z})\Big|\\
    & & +\,\Big|\int f\Big(\sum_{k=0}^h a^kz_{1-k}\,,\ldots,\sum_{k=0}^h a^kz_{n-k}\Big)\,\mu^{\otimes\Z}(d(z_j)_{j\in\Z})\\
    & & \quad-\int f\Big(\sum_{k=0}^\infty a^kz_{1-k}\,,\ldots,\sum_{k=0}^\infty a^kz_{n-k}\Big)\,\mu^{\otimes\Z}(d(z_j)_{j\in\Z}\Big|\\
    & =: & S_{2,1}(m,h)\,+\,S_{2,2}(m,h)\,+\,S_{2,3}(h).
\end{eqnarray*}
Fix $\varepsilon>0$, and note that $\sup_{m\in\N}\int|z_1|\,\mu_m(dz_1)<\infty$ (because $\mu_m$ converges weakly to $\mu$ and ${\cal M}_0$ was assumed to be locally uniformly $\psi_1$-integrating). Choose $h_0=h_0(\varepsilon)\in\N$ such that $\sum_{k=h_0+1}^\infty|a|^k\le(Ln\,\sup_{m\in\N}\int|z_1|\,\mu_m(dz_1))^{-1}\varepsilon/3$. By the Lipschitz continuity of $f$, we obtain
\begin{eqnarray}
    S_{2,1}(m,h_0)
    & \le & \int L\sum_{i=1}^n\Big|\sum_{k=h_0+1}^\infty a^kz_{i-k}\Big|\,\mu_m^{\otimes\Z}(d(z_j)_{j\in\Z})\nonumber\\
    & \le & (Ln\int|z_1|\,\mu_m(dz_1)\Big)\sum_{k=h_0+1}^\infty|a|^k\nonumber\\
    & \le & \varepsilon/3\quad\mbox{ for all }m\in\N.\label{hampel-huber generalized - Yule Walker - Lemma 2 - PROOF 10}
\end{eqnarray}
Analogously we obtain
\begin{eqnarray}\label{hampel-huber generalized - Yule Walker - Lemma 2 - PROOF 20}
    S_{2,3}(h_0)\,\le\,\varepsilon/3.
\end{eqnarray}
Further, $\mu_m$ converges weakly to $\mu$ and therefore $\mu_m^{\otimes (n+h_0)}$ converges weakly to $\mu^{\otimes (n+h_0)}$. Since the mapping $(z_j)_{j\in\{-h_0+1,\ldots,0,\ldots,n\}}\mapsto f(\sum_{k=0}^{h_0} a^kz_{1-k},\ldots,\sum_{k=0}^{h_0} a^kz_{n-k})$ is bounded and continuous on $\R^{\{-h_0+1,\ldots,0,\ldots,n\}}$, it follows that there exists some $m_0=m_0(\varepsilon)\in\N$ such that
\begin{eqnarray}\label{hampel-huber generalized - Yule Walker - Lemma 2 - PROOF 30}
    S_{2,2}(m,h_0)\,\le\,\varepsilon/3\quad\mbox{ for all }m\ge m_0.
\end{eqnarray}
By (\ref{hampel-huber generalized - Yule Walker - Lemma 2 - PROOF 10})--(\ref{hampel-huber generalized - Yule Walker - Lemma 2 - PROOF 30}), we have $S_2(m)\le\varepsilon$ for all $m\ge m_0$. Thus, $S_2(m)\to 0$.

We have shown that $\ex^{(a_m,\mu_m)}[f(X_1,\ldots,X_n)]\to\ex^{(a,\mu)}[f(X_1,\ldots,X_n)]$ as $m\to\infty$ for every bounded and Lipschitz continuous $f:\R^n\to\R$. This completes the proof.
\end{proof}

\bigskip

\begin{proof}{\bf of Theorem \ref{hampel-huber generalized - Yule Walker}:}
The sequence $(\widehat T_n)$ is asymptotically $(d_\Theta,\rho_{\mbox{\scriptsize{\rm P}}})$-robust on $\Theta_0$, because conditions (a)--(b) of Theorem \ref{hampel-huber generalized} are satisfied for $\Upsilon:=(-1,1)$, $U(a,\mu):=T(a,\mu):=a$, $\widehat U_n:=\widehat T_n$ and $V_n(u):=u$ for all $n\in\N$. Indeed, condition (a) trivially holds, because $(a,\mu)\mapsto a$ is continuous. Further, let $\eta>0$. For every $(a,\mu)\in\Theta$ with $a=0$ we have $\pr^{(a,\mu)}[|\widehat T_n(X_1,\ldots,X_n)-T(a,\mu)|\ge\eta]=\pr^{(a,\mu)}[0\ge\eta]=0=\pr^{(a,\mu)}[0\ge\eta]=\pr^{(a,\mu)}[|\frac{1}{n}\sum_{i=1}^{n}X_i^2-\ex^{(a,\mu)}[X_1^2]|\ge\eta]$. For every $(a,\mu)\in\Theta$ with $a\not=0$ we obtain
\begin{eqnarray*}
    \lefteqn{\pr^{(a,\mu)}\big[\big|\widehat T_n(X_1,\ldots,X_n)-T(a,\mu)\big|\ge\eta\big]}\\
    & = & \pr^{(a,\mu)}\Big[\,\Big|\frac{\frac{1}{n-1}\sum_{i=1}^{n-1}X_iX_{i+1}}{\frac{1}{n}\sum_{i=1}^{n}X_i^2}-\frac{\ex^{(a,\mu)}[X_1X_2]}{\ex^{(a,\mu)}[X_1^2]}\Big|\ge\eta\,\Big|\,\sum_{i=1}^{n}X_i^2>0\Big]\nonumber\\
    & \le & \pr^{(a,\mu)}\Big[\frac{1}{\ex^{(a,\mu)}[X_1^2]}\,\Big|\frac{1}{n-1}\sum_{i=1}^{n-1}X_iX_{i+1}-\ex^{(a,\mu)}[X_1X_2]\Big|\ge\eta/2\,\Big|\,\sum_{i=1}^{n}X_i^2>0\Big]\nonumber\\
    & & +\,\pr^{(a,\mu)}\Big[\frac{\frac{1}{n-1}\sum_{i=1}^{n-1}|X_iX_{i+1}|}{\big(\frac{1}{n}\sum_{i=1}^{n}X_i^2\big)\,\ex^{(a,\mu)}[X_1^2]}\,
    \Big|\frac{1}{n}\sum_{i=1}^{n}X_i^2-\ex^{(a,\mu)}[X_1^2]\Big|\ge\eta/2\,\Big|\,\sum_{i=1}^{n}X_i^2>0\Big]\nonumber\\
    & \le & \pr^{(a,\mu)}\Big[\Big|\frac{1}{n-1}\sum_{i=1}^{n-1}X_iX_{i+1}-\ex^{(a,\mu)}[X_1X_2]\Big|\ge\ex^{(a,\mu)}[X_1^2]\,\eta/2\Big]\nonumber\\
    & & +\,\pr^{(a,\mu)}\Big[\Big|\frac{1}{n}\sum_{i=1}^{n}X_i^2-\ex^{(a,\mu)}[X_1^2]\Big|\ge\ex^{(a,\mu)}[X_1^2]\,\eta/4\Big]\nonumber,
\end{eqnarray*}
where for first and the last step we used $\pr^{(a,\mu)}[\sum_{i=1}^{n}X_i^2=0]=0$ (recall that $\mu$ is absolutely continuous w.r.t.\ the Lebesgue measure) and for the last step we used H\"older's inequality in the form of $\frac{1}{n-1}\sum_{i=1}^{n-1}|X_iX_{i+1}|\le\frac{n}{n-1}\frac{1}{n}\sum_{i=1}^{n}X_i^2\le\frac{2}{n}\sum_{i=1}^{n}X_i^2$. Since the mapping
$$
    (a,\mu)\longmapsto\ex^{(a,\mu)}[X_1^2]=\sum_{k=0}^\infty\sum_{m=0}^\infty a^{k+m}\int z_{1-k}\,z_{1-m}\,\mu^{\otimes\Z}(d(z_j)_{j\in\Z})
$$
is easily seen to be $(d_\Theta,|\cdot|)$-continuous on $\Theta_0$ (use Lemma \ref{weak and psi weak topology} and the assumption that ${\cal M}_0$ is locally uniformly $\psi_2$-integrating), it follows from $\ex^{(a,\mu)}[X_1^2]>0$ (for all $(a,\mu)\in\Theta=(-1,1)\times{\cal M}_1^{0,2}(\R)$ with $a\not=0$) and the first part of Lemma \ref{hampel-huber generalized - Yule Walker - Lemma} that condition (b) of Theorem \ref{hampel-huber generalized} holds too.

The sequence $(\widehat T_n)$ is also finite-sample $(d_\Theta,\rho_{\mbox{\scriptsize{\rm P}}})$-robust on $\Theta_0$. Indeed, by Lemma \ref{hampel-huber generalized - Yule Walker - Lemma 2} the mapping $(a,\mu)\mapsto\pr^{(a,\mu)}\circ(X_1,\ldots,X_n)^{-1}$ is $(d_\Theta,\rho_{n})$-continuous, where $\rho_{n}$ refers to any metric on ${\cal M}_1(\R^n)$ which metrizes the weak topology. Moreover, the mapping $(x_1,\ldots,x_n)\mapsto \widehat t_n(x_1,\ldots,x_n):=\widehat T_n(x_1,\ldots,x_n)$ is $(\|\cdot\|,d_\Sigma)$-continuous on $\R^n\setminus\{(0,\ldots,0)\}$. For every $\theta\in\Theta_0$ the probability measure $\pr^{\theta}\circ(X_1,\ldots,X_n)^{-1}$ is absolutely continuous w.r.t.\ the $n$-dimensional Lebesgue measure, because the elements of ${\cal M}_0$ were assumed to be absolute continuous w.r.t.\ the Lebesgue measure. It follows that for every $\theta\in\Theta_0$ the $({\cal B}(\R^n),{\cal S})$-measurable mapping $(x_1,\ldots,x_n)\mapsto \widehat t_n(x_1,\ldots,x_n)$ is continuous outside a $\pr^{\theta}\circ(X_1,\ldots,X_n)^{-1}$-null set. Then, setting $\Pi_n:=(X_1,\ldots,X_n)$, finite-sample $(d_\Theta,\rho_{\mbox{\scriptsize{\rm P}}})$-robustness of $(\widehat T_n)$ can be obtained as in the proof of Theorem \ref{hampel-huber generalized - finite sample} (Section \ref{proof of hampel-huber generalized - finite sample}), where one has to use the Continuous Mapping theorem in the form of Theorem 2.7 in \cite{Billingsley1999}. This completes the proof of Theorem \ref{hampel-huber generalized - Yule Walker}.
\end{proof}

%%%%%%%%%%%%%%%%%%%%%%%%%%%%%%%%%%%%%%%%%%%%%%%%%%%%%%%%%%%%%%%%
%%%%%%%%%%%%%%%%%%%%%%%%%%%%%%%%%%%%%%%%%%%%%%%%%%%%%%%%%%%%%%%%
%%%%%%%%%%%%%%%%%%%%%%%%%%%%%%%%%%%%%%%%%%%%%%%%%%%%%%%%%%%%%%%%
%%%%%%%%%%%%%%%%%%%%%%%%%%%%%%%%%%%%%%%%%%%%%%%%%%%%%%%%%%%%%%%%
%%%%%%%%%%%%%%%%%%%%%%%%%%%%%%%%%%%%%%%%%%%%%%%%%%%%%%%%%%%%%%%%
%%%%%%%%%%%%%%%%%%%%%%%%%%%%%%%%%%%%%%%%%%%%%%%%%%%%%%%%%%%%%%%%

\section{Proofs of the results of Section \ref{Sec Hampel Huber}}\label{Proofs of the results of Section Hampel Huber}

The proof of Theorem \ref{hampel-huber generalized} relies on Strassen's theorem. For the reader's convenience, we first of all recall Strassen's theorem as formulated in Theorem 2.4.7 in \cite{Huber1981}; the proof is contained in the seminal paper \cite{Strassen1965}. See also Theorem 11.6.2 and the succeeding remark in \cite{Dudley2002}.

\begin{theorem}{\em (Strassen)}\label{strassens theorem}
Let $(\Sigma,d_\Sigma)$ be a complete and separable metric space equipped with the corresponding Borel $\sigma$-field ${\cal S}$. Then, for any two probability measures $\mu_1,\mu_2$ on $(\Sigma,{\cal S})$ and any $\alpha,\beta>0$, the following two statements are equivalent:
\begin{itemize}
    \item[(i)] For every $A\in{\cal S}$ we have
    $$
        \mu_1[A]\,\le\,\mu_2[A^{\beta}]+\alpha,
    $$
    where $A^{\beta}:=\{s\in\Sigma\,:\,\inf_{a\in A}d_\Sigma(s,a)\le\beta\}$.
    \item[(ii)] There is some probability measure $\mu$ on $(\Sigma\times \Sigma,{\cal S}\otimes{\cal S})$ such that $\mu\circ\pi_1^{-1}=\mu_1$, $\mu\circ\pi_2^{-1}=\mu_2$, and
    $$
        \mu\big[\big\{(s_1,s_2)\in\Sigma\times\Sigma:\,d_\Sigma(s_1,s_2)\le\beta\big\}\big]\,\ge\,1-\alpha,
    $$
    where $\pi_i:\Sigma\times\Sigma\to\Sigma$ denotes the projection on the $i$-th coordinate, $i=1,2$.
\end{itemize}
\end{theorem}

%%%%%%%%%%%%%%%%%%%%%%%%%%%%%%%%%%%%%%%%%%%%%%%%%%%%%%%%%%%%%%%%
%%%%%%%%%%%%%%%%%%%%%%%%%%%%%%%%%%%%%%%%%%%%%%%%%%%%%%%%%%%%%%%%
%%%%%%%%%%%%%%%%%%%%%%%%%%%%%%%%%%%%%%%%%%%%%%%%%%%%%%%%%%%%%%%%

\subsection{Proof of Theorem \ref{hampel-huber generalized} and Remark \ref{hampel-huber generalized - remark on b}}\label{proof of hampel-huber generalized}

We will adapt arguments of the proof of Theorem 2.21 in \cite{Huber1981}. We have to show that for every $\theta_1\in\Theta_0$ and $\varepsilon>0$ there are some $\delta>0$ and $n_0\in\N$ such that
\begin{equation}\label{proof of hampel-huber generalized - eq - 1}
    \theta_2\in\Theta_0,\quad d_\Theta(\theta_1,\theta_2)\le\delta\quad\Longrightarrow\quad \rho_{\mbox{\scriptsize{\rm P}}}(\pr^{\theta_1}\circ \widehat T_n^{-1},\pr^{\theta_2}\circ \widehat T_n^{-1})\le\varepsilon \quad\mbox{for all }n\ge n_0.
\end{equation}
Since
$$
    \rho_{\mbox{\scriptsize{\rm P}}}(\pr^{\theta_1}\circ \widehat T_n^{-1},\pr^{\theta_2}\circ \widehat T_n^{-1})\,\le\,
    \rho_{\mbox{\scriptsize{\rm P}}}(\pr^{\theta_1}\circ \widehat T_n^{-1},\delta_{T_n(\theta_1)})\,+\,
    \rho_{\mbox{\scriptsize{\rm P}}}(\delta_{T_n(\theta_1)},\pr^{\theta_2}\circ \widehat T_n^{-1})
$$
(with $\delta_{T_n(\theta_1)}$ the dirac measure on $(\Sigma,{\cal S})$ with atom $T_n(\theta_1)$), for (\ref{proof of hampel-huber generalized - eq - 1}) it suffices to show that for every $\theta_1\in\Theta_0$ and $\varepsilon>0$ there are some $\delta>0$ and $n_0\in\N$ such that
\begin{equation}\label{proof of hampel-huber generalized - eq - 2}
    \theta_2\in\Theta_0,\quad d_\Theta(\theta_1,\theta_2)\le\delta\quad\Longrightarrow\quad \rho_{\mbox{\scriptsize{\rm P}}}(\delta_{T_n(\theta_1)},\pr^{\theta_2}\circ \widehat T_n^{-1})\le\varepsilon/2 \quad\mbox{for all }n\ge n_0.
\end{equation}
The remainder of the proof is divided into two steps. In Step 1, we will verify that for (\ref{proof of hampel-huber generalized - eq - 2}) it suffices to show that for every $\theta_1\in\Theta_0$ and $\varepsilon>0$ there are some $\delta>0$ and $n_0\in\N$ such that
\begin{eqnarray}
    \lefteqn{\theta_2\in\Theta_0,\quad d_\Theta(\theta_1,\theta_2)\le\delta}\nonumber\\
    & & \Longrightarrow\quad \pr^{\theta_2}\big[\big\{\omega\in\Omega:\,d_\Sigma(T_n(\theta_1),\widehat T_n(\omega))\le\varepsilon/2\big\}\big]\ge 1-\varepsilon/2 \quad\mbox{for all }n\ge n_0.\label{proof of hampel-huber generalized - eq - 3}
\end{eqnarray}
In Step 2, we will verify (\ref{proof of hampel-huber generalized - eq - 3}).

{\em Step 1.} Note that the right-hand side in (\ref{proof of hampel-huber generalized - eq - 3}) is equivalent to
\begin{equation}\label{proof of hampel-huber generalized - eq - 3 - 5}
    \big(\delta_{T_n(\theta_1)}\times(\pr^{\theta_2}\circ \widehat T_n^{-1})\big)\big[\big\{(s_1,s_2)\in\Sigma\times\Sigma:\,d_\Sigma(s_1,s_2)\le\varepsilon/2\big\}\big]\ge 1-\varepsilon/2\quad\mbox{for all }n\ge n_0.
\end{equation}
By the implication (ii)$\Rightarrow$(i) in Strassen's theorem \ref{strassens theorem} (with $\mu:=\delta_{T_n(\theta_1)}\times(\pr^{\theta_2}\circ \widehat T_n^{-1})$ and $\alpha:=\beta:=\varepsilon/2$), condition (\ref{proof of hampel-huber generalized - eq - 3 - 5}) implies
$$
    \delta_{T_n(\theta_1)}[A]\,\le\,\pr^{\theta_2}\circ \widehat T_n^{-1}[A^{\varepsilon/2}]+\varepsilon/2\quad\mbox{ for all }A\in{\cal S},\mbox{ for all }n\ge n_0,
$$
that is,
$$
    \rho_{\mbox{\scriptsize{\rm P}}}(\delta_{T_n(\theta_1)},\pr^{\theta_2}\circ \widehat T_n^{-1})\,\le\,\varepsilon/2\quad\mbox{for all }n\ge n_0.
$$
Thus, the right-hand side in (\ref{proof of hampel-huber generalized - eq - 3}) implies the right-hand side in (\ref{proof of hampel-huber generalized - eq - 2}).

{\em Step 2.} To verify (\ref{proof of hampel-huber generalized - eq - 3}), fix $\theta_1\in\Theta_0$ and $\varepsilon>0$. By the $(d_\Theta,d_\Sigma)$-continuity of $U|_{\Theta_0}$ at $\theta_1$, we can find for every $\eta>0$ some $\delta'=\delta'(\eta)>0$ such that for every $\theta_2\in\Theta_0$,
\begin{equation}\label{proof of hampel-huber generalized - eq - 4}
    d_\Theta(\theta_1,\theta_2)\le \delta'\quad\Longrightarrow\quad d_\Upsilon(U(\theta_1),U(\theta_2))\le\eta.
\end{equation}
In particular, for any $\theta_2\in\Theta_0$ satisfying $d_\Theta(\theta_1,\theta_2)\le\delta'$ we have
\begin{eqnarray}
    d_\Upsilon(U(\theta_1),\widehat U_n(\cdot))
    & \le &  d_\Upsilon(U(\theta_1),U(\theta_2))\,+\,d_\Upsilon(U(\theta_2),\widehat U_n(\cdot))\nonumber\\
    & \le &  \eta\,+\,d_\Upsilon(U(\theta_2),\widehat U_n(\cdot)).\label{proof of hampel-huber generalized - eq - 5}
\end{eqnarray}
Further, due to assumption (b) we can find some $\delta''=\delta''(\eta)>0$ and $n_0'=n_0'(\eta)\in\N$ such that for all $n\ge n_0'$ and $\theta_2\in\Theta_0$ satisfying $d_\Theta(\theta_1,\theta_2)\le\delta''$,
\begin{equation}\label{proof of hampel-huber generalized - eq - 6}
    \pr^{\theta_2}\big[\big\{\omega\in\Omega:\,d_\Upsilon(\widehat U_n(\omega),U(\theta_2))\le\eta\big\}\big]\,\ge\,1-\varepsilon/2.
\end{equation}
By (\ref{proof of hampel-huber generalized - eq - 5}), the left-hand side in (\ref{proof of hampel-huber generalized - eq - 6}) is bounded above by $\pr^{\theta_2}[d_\Upsilon(\widehat U_n,U(\theta_1))\le 2\eta]$ for all $\theta_2\in\Theta_0$ satisfying $d_\Theta(\theta_1,\theta_2)\le\delta=\delta(\eta):=\min\{\delta',\delta''\}=\min\{\delta'(\eta),\delta''(\eta)\}$. That is, for all $n\ge n_0'$ and $\theta_2\in\Theta_0$ satisfying $d(\theta_1,\theta_2)\le\delta$,
\begin{equation}\label{proof of hampel-huber generalized - eq - 7}
    \pr^{\theta_2}\big[\big\{\omega\in\Omega:\,d_\Upsilon(\widehat U_n(\omega),U(\theta_1))\le 2\eta\big\}\big]\,\ge\, 1-\varepsilon/2.
\end{equation}
In the case where $\Upsilon=\Sigma$ and $V_n(u)=u$, $n\in\N$, we can choose $\eta=\eta(\varepsilon):=\varepsilon/4$ to obtain (\ref{proof of hampel-huber generalized - eq - 3}). This justifies Remark \ref{hampel-huber generalized - remark on b}. In the general case, we can conclude (\ref{proof of hampel-huber generalized - eq - 3}) as follows. By the asymptotic $(d_\Upsilon,d_\Sigma)$-continuity of $(V_n)$ we can find some $\eta=\eta(\varepsilon)$ and $n_0''=n_0''(\varepsilon)$ such that $d_\Upsilon(u,U(\theta_1))\le 2\eta$ implies $d_\Sigma(V_n(u),V_n(U(\theta_1)))\le\varepsilon/2$ for all $u\in\Upsilon$ and $n\ge n_0''$. Together with (\ref{proof of hampel-huber generalized - eq - 7}) and the representations $T_n=V_n\circ U$ and $\widehat T_n=V_n\circ\widehat U_n$, this implies (\ref{proof of hampel-huber generalized - eq - 3}) with $n_0=n_0(\varepsilon):=\max\{n_0'(\eta),n_0''\}$ and $\delta=\delta(\varepsilon):=\delta(\eta)$. This completes the proof of Theorem \ref{hampel-huber generalized}.{\hspace*{\fill}$\Box$\par\bigskip}

%%%%%%%%%%%%%%%%%%%%%%%%%%%%%%%%%%%%%%%%%%%%%%%%%%%%%%%%%%%%%%%%
%%%%%%%%%%%%%%%%%%%%%%%%%%%%%%%%%%%%%%%%%%%%%%%%%%%%%%%%%%%%%%%%
%%%%%%%%%%%%%%%%%%%%%%%%%%%%%%%%%%%%%%%%%%%%%%%%%%%%%%%%%%%%%%%%

\subsection{Proof of Theorem \ref{hampel-huber generalized - reversed}}\label{proof of hampel-huber generalized - reversed}

Let $\theta_1\in\Theta_0$ and $\varepsilon>0$ be fixed. By the triangular inequality, we have for every $\theta_2\in\Theta_0$
\begin{eqnarray*}
    \lefteqn{\min\{d_\Sigma(T_0(\theta_1),T_0(\theta_2));1\} }\\
    & = & \rho(\delta_{T_0(\theta_1)},\delta_{T_0(\theta_2)})\\
    & \le & \rho(\delta_{T_0(\theta_1)},\pr^{\theta_1}\circ \widehat T_n^{-1})+\rho(\pr^{\theta_1}\circ \widehat T_n^{-1},\pr^{\theta_2}\circ \widehat T_n^{-1})+  \rho(\pr^{\theta_2}\circ \widehat T_n^{-1},\delta_{T_0(\theta_2)}).
\end{eqnarray*}
By assumption ($\alpha$), we may choose $\delta>0$ and $n_0\in\N$ such that the second summand is bounded above by $\varepsilon/3$ for all $n\ge n_0$ and $\theta_2\in\Theta_0$ with $d_\Theta(\theta_1,\theta_2)\le\delta$. Since $\rho$ was assumed to metrize the weak topology, condition ($\beta$) implies that the third summand converges to $0$ as $n\to\infty$ for every $\theta_2\in\Theta_0$. That is, for every $\theta_2\in\Theta_0$ and sufficiently large $n=n(\theta_2)\ge n_0$ the third summand is bounded above by $\varepsilon/3$ too. The same argument shows that for sufficiently large $n=n(\theta_1)\ge n_0$ the first summand is bounded above by $\varepsilon/3$. Hence, we have found some $\delta>0$ for which $\min\{d_\Sigma(T_0(\theta_1),T_0(\theta_2));1\}\le\varepsilon$ holds for every $\theta_2\in\Theta_0$ with $d(\theta_1,\theta_2)\le\delta$. This finishes the proof.{\hspace*{\fill}$\Box$\par\bigskip}

%%%%%%%%%%%%%%%%%%%%%%%%%%%%%%%%%%%%%%%%%%%%%%%%%%%%%%%%%%%%%%%%
%%%%%%%%%%%%%%%%%%%%%%%%%%%%%%%%%%%%%%%%%%%%%%%%%%%%%%%%%%%%%%%%
%%%%%%%%%%%%%%%%%%%%%%%%%%%%%%%%%%%%%%%%%%%%%%%%%%%%%%%%%%%%%%%%

\subsection{Proof of Theorem \ref{hampel-huber generalized - finite sample}}\label{proof of hampel-huber generalized - finite sample}

We will adapt arguments of the proof of Theorem 2 in \cite{Cuevas1988}. Finite-sample robustness of $(\widehat T_n)$ means that for every $\theta_1\in\Theta_0$, $\varepsilon>0$, and $n\in\N$ there is some $\delta>0$ such that $\rho(\pr^{\theta_1}\circ\widehat T_n^{-1},\pr^{\theta_2}\circ\widehat T_n^{-1})\le\varepsilon$ for every $\theta_2\in\Theta_0$ with $d_\Theta(\theta_1,\theta_2)\le\delta$. That is, we have to show that the mapping $\Theta_0\ni\theta\mapsto\pr^{\theta}\circ\widehat T_n^{-1}$ is weakly continuous at every $\theta_1\in\Theta_0$. It suffices to show that this mapping is sequentially weakly continuous at every $\theta_1\in\Theta_0$. Let $\theta_1\in\Theta_0$ and $(\theta_{2,k})\subset\Theta_0$ be any sequence such that $\lim_{k\to\infty}d_\Theta(\theta_1,\theta_{2,k})=0$. By assumption (d) we then have that $\pr^{\theta_{k,2}}\circ\Pi_n^{-1}$ converges weakly to $\pr^{\theta_{1}}\circ\Pi_n^{-1}$ as $k\to\infty$. It follows by assumption (c) and the Continuous Mapping theorem that $(\pr^{\theta_{k,2}}\circ\Pi_n^{-1})\circ\widehat t_n^{-1}=\pr^{\theta_{k,2}}\circ\widehat T_n^{-1}$ converges weakly to $(\pr^{\theta_{1}}\circ\Pi_n^{-1})\circ\widehat t_n^{-1}=\pr^{\theta_{1}}\circ\widehat T_n^{-1}$ as $k\to\infty$. This completes the proof of Theorem \ref{hampel-huber generalized - finite sample}.
{\hspace*{\fill}$\Box$\par\bigskip}

%%%%%%%%%%%%%%%%%%%%%%%%%%%%%%%%%%%%%%%%%%%%%%%%%%%%%%%%%%%%%%%%
%%%%%%%%%%%%%%%%%%%%%%%%%%%%%%%%%%%%%%%%%%%%%%%%%%%%%%%%%%%%%%%%
%%%%%%%%%%%%%%%%%%%%%%%%%%%%%%%%%%%%%%%%%%%%%%%%%%%%%%%%%%%%%%%%
%%%%%%%%%%%%%%%%%%%%%%%%%%%%%%%%%%%%%%%%%%%%%%%%%%%%%%%%%%%%%%%%
%%%%%%%%%%%%%%%%%%%%%%%%%%%%%%%%%%%%%%%%%%%%%%%%%%%%%%%%%%%%%%%%
%%%%%%%%%%%%%%%%%%%%%%%%%%%%%%%%%%%%%%%%%%%%%%%%%%%%%%%%%%%%%%%%

%\section*{Acknowledgement}
%
%{\color{blue}
%I thank an anonymous reviewer for comments and suggestions that led to an improvement of the exhibition and the readability of the paper.}
%%The research leading to these results has received funding from the European Union Seventh Framework Programme (FP7/2007-2013) under
%The research leading to the results of this paper has received funding from the BMBF project HYPERMATH under grant 05M13TSC.
%
%%The support by the BMBF project HYPERMATH under grant 05M13LBA is also gratefully acknowledged.

%%%%%%%%%%%%%%%%%%%%%%%%%%%%%%%%%%%%%%%%%%%%%%%%%%%%%%%%%%%%%%%%
%%%%%%%%%%%%%%%%%%%%%%%%%%%%%%%%%%%%%%%%%%%%%%%%%%%%%%%%%%%%%%%%
%%%%%%%%%%%%%%%%%%%%%%%%%%%%%%%%%%%%%%%%%%%%%%%%%%%%%%%%%%%%%%%%
%%%%%%%%%%%%%%%%%%%%%%%%%%%%%%%%%%%%%%%%%%%%%%%%%%%%%%%%%%%%%%%%
%%%%%%%%%%%%%%%%%%%%%%%%%%%%%%%%%%%%%%%%%%%%%%%%%%%%%%%%%%%%%%%%
%%%%%%%%%%%%%%%%%%%%%%%%%%%%%%%%%%%%%%%%%%%%%%%%%%%%%%%%%%%%%%%%

%%%%%%%%%%%%%%%%%%%%%%%%%%%%%%%%%%%%%%%%%%%%%%%%%%%%%%%%%%%%%%%%
%%%%%%%%%%%%%%%%%%%%%%%%%%%%%%%%%%%%%%%%%%%%%%%%%%%%%%%%%%%%%%%%
%%%%%%%%%%%%%%%%%%%%%%%%%%%%%%%%%%%%%%%%%%%%%%%%%%%%%%%%%%%%%%%%

\end{document}